\documentclass[smallextended]{svjour3}
\usepackage[a4paper]{geometry}
\usepackage[utf8]{inputenc}
\usepackage[T1]{fontenc}

  \usepackage{amsmath}
  \usepackage{amsfonts}
  \usepackage{amssymb}

\usepackage{dsfont}
\usepackage{bm}
\usepackage{enumerate}
\usepackage{graphicx}
\usepackage{paralist}
\usepackage{authblk}

\usepackage{natbib}
\usepackage{color}
\usepackage{float}

\usepackage{color}

\definecolor{asparagus}{rgb}{0.53, 0.66, 0.42}
\definecolor{olive}{rgb}{0.5, 0.5, 0.0}
\definecolor{antiquefuchsia}{rgb}{0.57, 0.36, 0.51}
\definecolor{golden(brown)}{rgb}{0.6, 0.4, 0.08} %
\definecolor{gray-asparagus}{rgb}{0.27, 0.35, 0.27}
\definecolor{glaucous}{rgb}{0.38, 0.51, 0.71}
\definecolor{airforceblue}{rgb}{0.36, 0.54, 0.66}
\definecolor{blue(munsell)}{rgb}{0.0, 0.5, 0.69}
\definecolor{americanrose}{rgb}{1.0, 0.01, 0.24}

\usepackage{amsthm}
\theoremstyle{plain}
\newtheorem{condition}[theorem]{Condition}

\numberwithin{equation}{section}

\newcommand{\diff}{\mathrm{d}} 
\newcommand{\reals}{\mathbb{R}} 
\newcommand{\dto}{\xrightarrow{d}} 

\newcommand{\point}{\,\cdot\,}

\newcommand{\GPr}{GP$_R$}
\newcommand{\GPS}{GP$_S$}
\newcommand{\GPs}{GP$_U$}
\newcommand{\GPT}{GP$_T$}

\newcommand{\abs}[1]{\lvert{#1}\rvert}  

\renewcommand{\Pr}{\operatorname{P}}
\newcommand{\Exp}{\operatorname{E}}
\newcommand{\expec}{\Exp}

\newcommand{\GP}{\operatorname{GP}}

\newcommand{\bzero}{\bm{0}}
\newcommand{\bone}{\bm{1}}
\newcommand{\binfty}{\bm{\infty}}
\newcommand{\1}{\mathds{1}}

\newcommand{\balpha}{{\bm{\alpha}}}

\newcommand{\bgamma}{{\bm{\gamma}}}

\newcommand{\bEta}{{\bm{\eta}}}
\newcommand{\bmu}{{\bm{\mu}}}
\newcommand{\bsigma}{{\bm{\sigma}}}

\newcommand{\bomega}{{\bm{\omega}}}

\newcommand{\ba}{\bm{a}}
\newcommand{\bb}{\bm{b}}

\newcommand{\bs}{\bm{s}}
\newcommand{\bu}{\bm{u}}
\newcommand{\bv}{\bm{v}}

\newcommand{\bx}{\bm{x}}
\newcommand{\by}{\bm{y}}
\newcommand{\bz}{\bm{z}}

\newcommand{\bR}{\bm{R}}
\newcommand{\bS}{\bm{S}}
\newcommand{\bT}{\bm{T}}
\newcommand{\bU}{\bm{U}}

\newcommand{\bW}{\bm{W}}
\newcommand{\bX}{\bm{X}}
\newcommand{\bY}{\bm{Y}}

\journalname{Extremes}

\title{Multivariate peaks over thresholds models}

\author{Holger Rootz\'{e}n \and Johan Segers \and Jennifer L. Wadsworth}

\institute{Holger Rootz\'{e}n     \at  \small Chalmers and Gothenburg University:   \email{hrootzen@chalmers.se}\\
Johan Segers \at \small Universit\'e catholique de Louvain: \email{johan.segers@uclouvain.be}\\
Jennifer L. Wadsworth \at \small Lancaster University: \email {j.wadsworth@lancaster.ac.uk}}

\begin{document}

\maketitle

\begin{abstract}
Multivariate peaks over thresholds modelling based on generalized Pareto distributions has up to now only been used in few and mostly two-dimensional situations. This paper contributes theoretical understanding, models which can respect physical constraints, inference tools, and simulation methods to support routine use, with an aim at higher dimensions.  We derive a general point process model for extreme episodes in data, and show how conditioning the distribution of extreme episodes on threshold exceedance gives four basic representations of the family of generalized Pareto distributions. The first representation is constructed on the real scale of the observations. The second one starts with a model on a standard exponential scale which is then transformed to the real scale. The third and fourth representations are reformulations of  a spectral representation proposed in A. Ferreira and L. de Haan [Bernoulli 20 (2014) 1717--1737]. Numerically tractable forms of densities and censored densities are found and give tools for flexible parametric likelihood inference. New simulation algorithms, explicit formulas for probabilities and conditional probabilities, and conditions which make the conditional distribution of weighted component sums generalized Pareto are derived.
\keywords{Extreme values, multivariate generalized Pareto distribution, peaks over threshold likelihoods, simulation of extremes}
 \subclass{62G32 \and 60G70 \and 62E10}
\end{abstract}

\section{Introduction}

Peaks over thresholds (PoT) modelling was introduced in the hydrological literature \citep{nerc1975}. The philosophy is simple: extreme events, perhaps extreme water levels, are often quite different from ordinary everyday behaviour, and ordinary behaviour then has little to say about extremes, so that only other extreme events give useful information about future extreme events. To make this idea operational, one defines an extreme event as a value, say a water level, which exceeds some high threshold, and only uses the sizes of the excesses over this threshold, the ``peaks over the threshold'', for statistical inference. This idea was given a theoretical foundation by combining it with asymptotic arguments motivating that the natural model is that exceedances occur according to a Poisson process and that excess sizes follow a generalized Pareto (GP) distribution \citep{balkema+dh:1974, pickands:1975, smith:1984, davison+s:1990}.

Since then, numerous papers have used one-dimensional PoT models (though often not under this name), in areas ranging from earth and atmosphere science to finance, see e.g.\ \citet{kysely+p+b:2010}, \citet{katz+parlange+naveau:2002}, and \citet{mcneil+f+e:2015}. The method has also been presented in a number of books, see e.g.\  \citet{coles2001}, \citet{beirlant2004}, and \citet{dey-yant2015}.

However, often it is not just one extreme event which is important, but an entire extreme episode. In the 2005 flooding of New Orleans caused by windstorm Katrina, more than 50 levees were breached. However, many others held, and damage was determined by which levees  held and which were flooded \citep{katrina2003}. Extreme rain can lead to devastating landslides, and can be caused by one day with very extreme rainfall, or by two or more consecutive days with smaller, but still extreme rain amounts  \citep{Guzzetti+p+r+s:2015}. The 2003 heat-wave in central Europe is estimated to have killed between 25\,000 and 70\,000 people. Many deaths, however, were not caused by one extremely hot day, but rather by a long sequence of high minimum nightly temperatures which led to increasing fatigue and eventually to death \citep{Grynszpan2003}. These and very many other important societal problems underline the importance of statistical methods which can handle multivariate extreme episodes.

 Using the same philosophy as for extreme events in one dimension, PoT modelling of extreme episodes proceeds by choosing a high threshold for each component of the episode, and then to consider an episode as extreme if at least one component exceeds its threshold. One then only models the difference between the values of the components and their respective thresholds. However, in the multivariate case all the componentwise differences in an extreme episode are  modelled, both the overshoots and the undershoots. For instance, in a rainfall episode affecting a number of catchments,  both the amount of rain in the catchments where rainfall exceeds the threshold and in catchments where the threshold is not exceeded are important. Additionally, the inclusion of undershoots increases the amount of information that can be used for inference. Just as in one dimension, the natural model is that extreme episodes occur according to a Poisson process and that overshoots and undershoots (or undershoots larger than a censoring threshold) jointly follow a multivariate GP distribution.

 The aim of this paper is to contribute  probabilistic understanding, physically motivated models, likelihood tools, and simulation methods, all of which are needed for multivariate PoT modelling of extreme episodes via multivariate GP distributions. Specifically, the key contributions are: new representations of GP distributions conducive to model construction; density formulas for each of these representations; new properties of multivariate GP distributions; and simulation tools. Many of these results are oriented towards enabling improved statistical modelling, but here we restrict ourselves to a probabilistic study. A companion paper \citep{kiriliouk+rootzen+segers+wadsworth:2016} addresses practical modelling aspects.

We begin by deriving the basic properties of the class of multivariate GP distributions. We then pursue the following program:
\begin{compactenum}[(i)]
\item
to exhibit the possible point process limits of extreme episodes in data;
\item
to show how conditioning on threshold exceedances transforms the distribution of the extreme episodes to GP distributions, and to use this to find physically motivated representations of the multivariate GP distributions; and
\item
to derive likelihoods and censored likelihoods for the representations in (ii).
\end{compactenum}

In part (ii) of the program, we develop four representations.  The first one is in the same units as the observations, i.e., on the real scale,  and in the second one  the model is built on a standard exponential scale and then transformed to the real observation scale. The third is a spectral representation proposed in \citet{ferreira2014}, and the fourth one a simple reformulation of this representation aimed at aiding model construction. A useful, and to us surprising, discovery is that it is possible to derive the density also for the fourth representation, and that this density in fact is simpler than the densities for the other two first representations.  The importance of (iii) is that likelihood inference makes it possible to incorporate covariates, e.g.\ temporal or spatial trends, in a flexible and practical way.

The insights and results obtained in carrying out this program, we believe, will lead to new models, new computational techniques, and new ways  to make the necessary compromises between modelling realism and computational tractability which together will make possible routine use, also in dimensions higher than two. The limiting factor is the number of parameters rather than the number of variables. The models mentioned in Example~\ref{ex:nuggets} may be a case in point. The formulas for probabilities, conditional probabilities and conditional densities given in Sections~\ref{sec:densities} and~\ref{sec:condprob}, together with the discovery that weighted sums of components of GP distributions conditioned to be positive also have a GP distribution, add to the usefulness of the methods. Simulation of GP distributions is needed for several reasons, including computation of the probabilities of complex dangerous events and goodness of fit checking. The final contribution of this paper is a number of simulation algorithms for multivariate GP distributions.

The multivariate GP distributions were introduced in \cite{tajvidi1996}, \citet[Chapter~8]{beirlant2004}, and \cite{rootzen2006}; see also \citet[Chapter~5]{falk2010}. A closely related approximation was used in \citet{smith+tawn+coles:1997}. The literature on applications of multivariate PoT modelling is rather sparse  \citep{brodin2009,  michel2009, aulbach2012multivariate1}. Some earlier papers use point process models which are closely related to the PoT/GP approach \citep{coles1991, joe+s+w:1992}. Other papers consider nonparametric or semiparametric rank-based PoT methods focusing on the dependence structure but largely ignoring modelling the margins \citep{dehaan2008, einmahl2012, einmahl2016}. However, the GP approach has the advantages that it provides complete models for the threshold excesses, that it can use well-established model checking tools, and that, compared to the point process approach, it leads to more natural parametrizations of trends in the Poisson process which governs the occurrence of extreme episodes.

There is an important literature on modelling componentwise, perhaps yearly, maxima  with multivariate generalized extreme value (GEV) distributions: for a survey in the spatial context see \cite{davison+p+r:2012}.  However, componentwise maxima may occur at different times for different components, and in many situations the focus is on the PoT structure: extremes which occur simultaneously. Additionally, likelihood inference for GEV distributions is complicated by a lack of tractable analytic expressions for high-dimensional densities, so that inference often is much easier, and perhaps more efficient, in GP models; see \cite{huser+d+g:2015} for a survey and an extensive comparison. The most important special case of GP models are those for which  all variables can be simultaneously extreme, and there is no mass placed on hyperplanes (see Section~\ref{sec:background} for details of the support); this is a typical modelling assumption. Further comment on the situation of asymptotic independence, where this does not hold, is made in Section~\ref{sec:conclusion}, as well as in \cite{kiriliouk+rootzen+segers+wadsworth:2016}.

Section~\ref{sec:background} derives and exemplifies the basic properties of the GP cumulative distribution functions (cdf-s). In Section~\ref{sec:point processes} we develop a point process model of extreme episodes, and  Section~\ref{sec:representation} shows how conditioning on exceeding high thresholds leads to three basic representations of the GP distributions. Section~\ref{sec:densities} exhibits the fourth representation and derives densities and censored likelihoods, while Section~\ref{sec:condprob} gives formulas for probabilities and conditional probabilities in GP distributions. Finally, Section~\ref{sec:simulation} contributes simulation algorithms for multivariate GP distributions and Section~\ref{sec:conclusion} discusses parametrization issues and gives a concluding overview.

\section{Multivariate generalized Pareto distributions}
\label{sec:background}

This section first briefly recalls and  adapts existing theory for multivariate GEV distributions, and then derives a number of the basic properties of GP distributions.

Throughout we use notation as follows. The maximum and minimum operators are denoted by the symbols $\vee$ and $\wedge$, respectively. Bold symbols denote $d$-variate vectors. For instance, $\balpha = (\alpha_1, \ldots, \alpha_d)$ and $\bzero = (0, \ldots, 0) \in \reals^d$. Operations and relations involving such vectors are meant componentwise, with shorter vectors being recycled if necessary. For instance $\ba\bx+\bb =(a_1x_1 + b_1, \ldots, a_dx_d + b_d)$, $\bx \le \by$ if $x_j \le y_j$ for $j = 1, \ldots, d$, and $t^\bgamma = (t^{\gamma_1}, \ldots, t^{\gamma_d})$.  If $F$ is a cdf then we write $\bar{F}=1-F$ for its tail function, and also write $F$ for the probability distribution determined by the cdf. That $\bX \sim F$ means that $\bX$  has distribution $F$, and $\dto$ denotes convergence in distribution. The symbol $\1$ is the indicator function: $\1_A$ equals $1$ on the set $A$ and $0$ otherwise.

For fixed $\gamma \in \reals$, the functions $x \mapsto (x^\gamma - 1)/\gamma$ (for $x > 0$) and $x \mapsto (1 + \gamma x)^{1/\gamma}$ are to be interpreted as their limits $\log(x)$ and $\exp(x)$, respectively, if $\gamma = 0$. This convention also applies componentwise to expressions of the form $(\bx^\bgamma - 1) / \bgamma$ and $(1 + \bgamma \bx)^{1/\bgamma}$.

Below we repeatedly use that if $\bX$ is a $d$-dimensional vector with $P(\bX \nleq \bu)>0$ and $\bs >\bzero$  then
\begin{equation}
\label{eq:conditioning}
  \Pr[\bs(\bX - \bu) \leq \bx \mid \bX - \bu \nleq \bzero ]
  =
  \frac{\Pr[\bX \leq \bx/\bs + \bu] - \Pr[\bX \leq (\bx\wedge \bzero)/\bs + \bu ]}{\Pr[\bX \nleq \bu]}.
\end{equation}

\subsection{Background: multivariate generalized extreme value distributions}
\label{subsec:GEV}

Throughout, $G$ denotes a $d$-variate GEV distribution, so that in particular $G$ has non-degenerate margins. The class of GEV distributions has the following equivalent characterizations, see e.g. \citet{beirlant2004}: (M1) {\em It is the class of limit distributions of location-scale normalized maxima}, i.e., the distributions which are limits
\begin{equation}
\label{eq:DA}
\textstyle \Pr[\ba_n^{-1} (\bigvee_{i=1}^d \bX_i - \bb_n) \leq \bx ] \dto G(\bx), \;\;\;  \text{as}  \;\;\; n \to \infty,
\end{equation}
of normalized maxima of independent and identically distributed (i.i.d.) vectors $ \bX_1, \bX_2, \ldots \sim F$,  for $\ba_n > \bzero$ and $\bb_n$;
 and (M2) {\em  It is the class of max-stable distributions}, i.e., distributions such that taking maxima of i.i.d.\ vectors from the distribution only leads to a location-scale change of the distribution.
By (M1) the class of GEV distributions is closed under location and scale changes.

The marginal distribution functions, $G_1, \ldots, G_d$, of $G$ may be written as

\begin{equation}
\label{eq:gmarginals}
  G_j(x) = \exp \left\{ - \left( 1 + \gamma_j \tfrac{x - \mu_j}{\alpha_j} \right)^{-1/\gamma_j} \right\},
\end{equation}
for $x \in \reals$ such that $\alpha_j + \gamma_j (x - \mu_j) > 0$. We will use this parametrization throughout.  The parameter range is $(\gamma_j, \mu_j, \alpha_j) \in \reals \times \reals \times (0, \infty)$. Define
\begin{equation*}
  \bsigma =  \balpha - \bgamma \bmu,
\end{equation*}
so that $\sigma_j = \alpha_j - \gamma_j \mu_j, \,j \in \{1, \ldots, d\}$. Then $G_j$ is supported by the interval
\begin{equation}
\label{eq:Gsupport}
  \tilde{I}_j
  =
  \begin{cases}
     (-\sigma_j/\gamma_j, \infty) & \text{if $\gamma_j > 0$,} \\
     (-\infty, \infty) & \text{if $\gamma_j = 0$,} \\
    (-\infty, -\sigma_j/\gamma_j) & \text{if $\gamma_j < 0$,}
  \end{cases}
\end{equation}
while $G$ is supported by a (subset of) the rectangle $ \tilde{I}_1 \times \cdots \times   \tilde{I}_d$. The lower and upper endpoints of $G_j$ are denoted by $\eta_j \in \reals \cup \{-\infty\}$ and $\omega_j \in \reals \cup \{+\infty\}$, respectively. One may alternatively write the condition $\bx \in \tilde{I}_1 \times \cdots \times   \tilde{I}_d$ as $\bgamma \bx + \bsigma > \bzero$.

Below we assume that $0 < G(\bzero) < 1$. This inequality is equivalent to $G_j(0) > 0$ for \emph{all} $j \in \{1, \ldots, d\}$ and $G_j(0) < 1$ for \emph{some} $j \in \{1, \ldots, d\}$. The equivalence follows from positive quadrant dependence, $G( \bzero ) \ge \prod_{j=1}^d G_j(0)$ \citep{marshall1983}. PoT models are determined by  the difference between the thresholds and the location parameters of the observations, and not by their individual values. Hence, it does not entail any loss of generality to shift the location parameters $\{\mu_i\}$ to make the assumption  $0 < G(\bzero) < 1$ hold.

We will often use the stronger condition that $\bsigma > \bzero$, i.e., that $\sigma_j > 0$ for \emph{all} $j \in \{1, \ldots, d\}$. By \eqref{eq:Gsupport}, this is equivalent to assuming that $0$ is in the interior of the support of every one of the $d$ margins $G_1, \ldots, G_d$, i.e., that $\eta_j < 0 < \omega_j$ and thus $0 < G_j(0) < 1$ for all $j \in \{1, \ldots, d\}$. This is an additional restriction only for $\gamma_j < 0$: if $\gamma_j = 0$, then $\sigma_j = \alpha_j > 0$, while if $\gamma_j > 0$ then $G(\bzero) > 0$ implies $\eta_j < 0$ and thus $\sigma_j = - \gamma_j \eta_j> 0$.

An easy argument shows that $G$ is max-stable if and only if for each $t>0$ there exist scale and location vectors $\ba_t \in (0, \infty)^d$ and $\bb_t \in \reals^d$ such that $G(\ba_t \bx +\bb_t)^t \equiv G(\bx)$ \citep[Equation~(5.17)]{resnick1987}. It follows from \eqref{eq:gmarginals} that these parameters are given by
\begin{equation} \label{eq:tparameters}
   \ba_t = t^{\bgamma}, \;\;\;\;
   \bb_t = \bsigma(t^{\bgamma} - \bone)/\bgamma.
\end{equation}

To a GEV distribution $G$ we can associate a Borel measure $\nu$ on $\prod_{j=1}^d [-\eta_j, \infty) \setminus \{ \bEta \}$ by the formula $\nu( \{ \by : \by \not\leq \bx \} ) = - \log G( \bx )$ for $\bx \in [-\infty, \infty)^d$, with the convention that $-\log(0) = \infty$ \citep[Proposition~5.8]{resnick1987}. The measure $\nu$ is called \emph{intensity measure} because, by (M1),
the limit of the expected number of location-scale normalized points, say
$\ba_n^{-1} ( \bX_i - \bb_n )$, $i \in \{1, \ldots, n\}$, in a Borel set $A$ which is bounded away from $\bEta$ and such that $\nu(\partial A) = 0$, is equal to $\nu(A)$. The intensity measure $\nu$ determines the limit distribution of the sequence of point processes $\sum_{i=1}^n \delta_{ \ba_n^{-1} ( \bX_i - \bb_n ) }$, see Section~\ref{sec:point processes}.

\subsection{Generalized Pareto distributions}
\label{subsec:GPD}

Let $G$ be a GEV distribution with $0 < G(\bzero) < 1$ and let $\nu$ be the corresponding intensity measure. Then $0 < \nu ( \{ \by : \by \not\leq \bzero \} ) < \infty$, so that we can define a probability measure supported by the set $\{ \by : \by \not\leq \bzero \}$ by restricting the intensity measure $\nu$ to that set and normalizing it. The result is the \emph{generalized Pareto (GP)} distribution associated to $G$. Its cdf $H$ may be expressed as
\begin{equation}
\label{eq:MGPD}
  H( \bx ) =
  \begin{cases}
    \dfrac{1}{\log G( \bzero )} \log \left( \dfrac{ G( \bx \wedge \bzero ) }{ G( \bx ) } \right)
    & \text{if $\bx > \bEta$,} \\[1em]
    0
    & \text{if $x_j < \eta_j$ for some $j=1,\ldots,d$,}
  \end{cases}
\end{equation}
see \citet[Chapter~8]{beirlant2004} and \cite{rootzen2006}. If a GEV cdf $G$ and a GP cdf $H$ satisfy \eqref{eq:MGPD}, then we say that they are \emph{associated} and write $H \leftrightarrow G$.
For completeness, we prove \eqref{eq:MGPD} in the Appendix. For points $\bx \in [-\infty, \infty)^d$ with $\bx \ge \bEta$ and $x_j = \eta_j$  for some $j$, the value of $H(\bx)$ is determined by right-hand continuity. Below is shown that  $\bEta$ is determined by the values of $ H( \bx )$ for $ \bx \geq \bzero$.

The probability that the $j$-th component, $j \in \{1, \ldots, d\}$, exceeds zero is equal to $1 - H_j(0) = \log G_j(0) / \log G(\bzero)$, which is positive if and only if $G_j(0) < 1$, that is, when $\sigma_j = \alpha_j - \gamma_j \mu_j > 0$. Since $G(\bzero) < 1$ implies that $G_j(0) < 1$ for some but not necessarily all $j$, the GP family includes distributions for which one (or several) of the components never exceed their threshold, so that the support of that component lies in  $[-\infty, 0]$.  This could be useful in some modelling situations, but still, the situation of main interest is when all components have a positive probability of being an exceedance, or equivalently when $H_j(0) < 1$ for all $j \in \{1, \ldots, d\}$, or, again equivalently, when $\bsigma > \bzero$.

Similarly to the characterizations (M1) and (M2) of the GEV distributions, the class of GP distributions $H$ such that $H_j(0) < 1$ for all $j \in \{1, \ldots, d\}$ has the following characterizations \citep{rootzen2006}\footnote{In the article, the truncation factor ``$\mbox{} \vee \bEta$'' is missing in Theorem 2.2 and Theorem 2.3~(ii). A correction note is forthcoming.}. The  functions  $\bsigma_t, \bu_t$ in the characterizations are assumed to be continuous, and additionally $\ \bu_t$ is assumed increasing.
\begin{enumerate}[(T1)]
\item
\emph{The GP distributions are limits of distributions of threshold excesses:} Let $\bX \sim F$. If there exist scaling and threshold functions $\bs_t \in (0, \infty)^d$ and $\bu_t \in \reals^d$ with  $F(\bu_t) < 1$ and $F(\bu_t) \to 1$ as $t \to \infty$, such that
\[
  \Pr[\bs_t^{-1}(\bX-\bu_t)\vee \bzero \leq \,\cdot\, \mid  \bX \nleq \bu_t]
  \dto H_+,
  \qquad \text{as}\ t \to \infty,
\]
 for some cdf $H_+$ with nondegenerate margins, then the function $\{H_+(\bx); \bx > \bzero\}$ can be uniquely extended to a GP cdf $H(\bx); \bx \in  \reals^d$, and

\begin{equation}\label{eq:GPlimit}
  \Pr[\bs_t^{-1}(\bX-\bu_t)\vee \bEta \leq \,\cdot\, \mid  \bX \nleq \bu_t]
  \dto H,
  \qquad \text{as}\ t \to \infty.
\end{equation}
\item
 \emph{The GP distributions are threshold-stable:} Let $\bX \sim H$ where $H$ has nondegenerate margins on $\reals_+$. If there exist scaling and threshold functions $\bs_t  \in (0, \infty)^d$ and $\bu_t \in \reals^d$,  with $\bu_1=\bzero$ and $H(\bu_t) \to 1$  as $t \to \infty$, such that
\begin{equation}\label{eq:equality}
   \Pr[\bs_t^{-1}(\bX-\bu_t) \leq \bx \mid  \bX \nleq \bu_t]
  = H(\bx)
\end{equation}
for $\bx \geq \bzero$ then there is an uniquely determined GP cdf $\tilde{H}$ such that $\tilde{H}(\bx)=H(\bx)$  for $\bx > \bEta$. Conversely, all GP distributions $H$ for which $H_j(0) < 1$ for all $j \in \{1, \ldots, d\}$ satisfy \eqref{eq:equality} for all $\bx \in \reals^d$.
\end{enumerate}

We use the term ``threshold-stable'' for property (T2) in analogy with the terms ``sum-stable'' and ``max-stable''. A distribution is sum- or max-stable if the sum or maximum, respectively, of independent variables with this distribution has the same distribution, up to a location-scale change. Analogously, a distribution is threshold-stable if conditioning on the exceedance of suitable higher thresholds leads to distributions which, up to scale changes,  are the same as the original distribution. This property is illustrated in Figure~\ref{fig:stability}, with $\bm{u}_t$ and $\bm{s}_t$ as given below in Theorem~\ref{prop:GPproperties}(viii).

\begin{figure}
 \centering
 \includegraphics[width=0.3\textwidth]{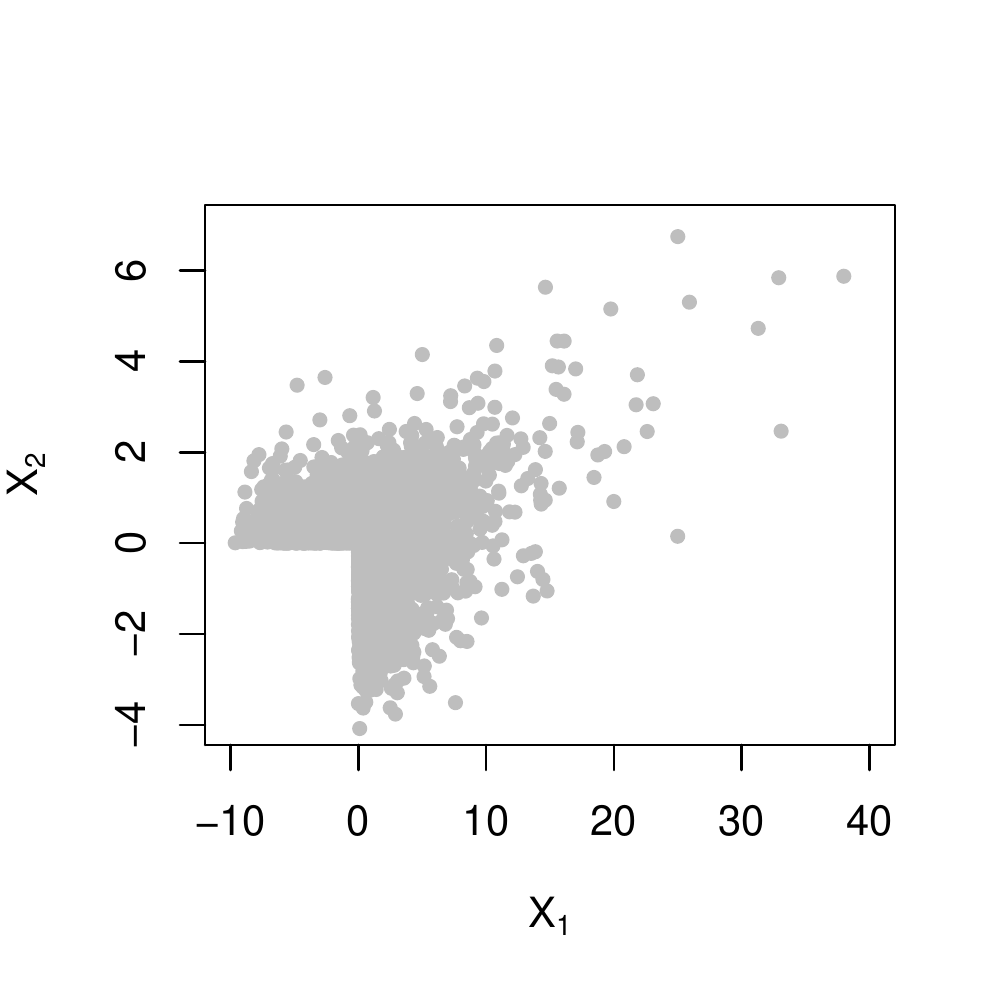}
 \includegraphics[width=0.3\textwidth]{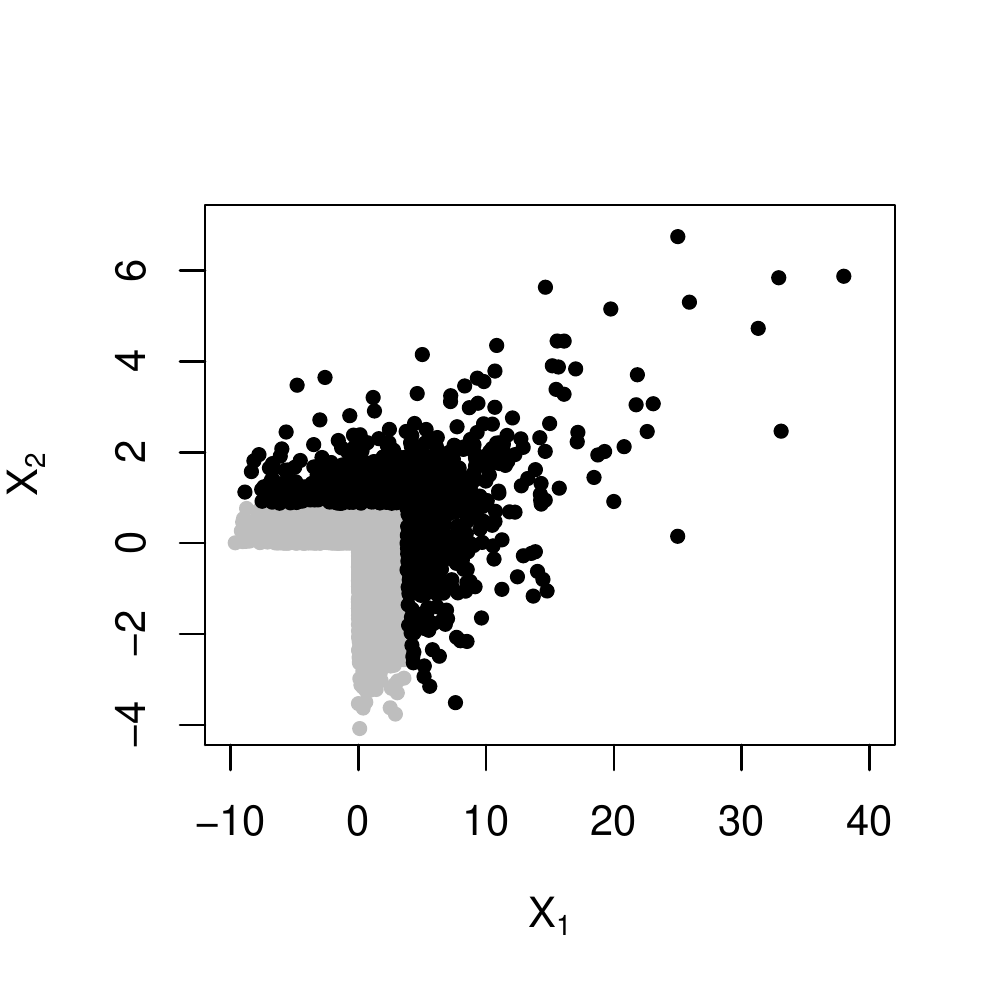}
 \includegraphics[width=0.3\textwidth]{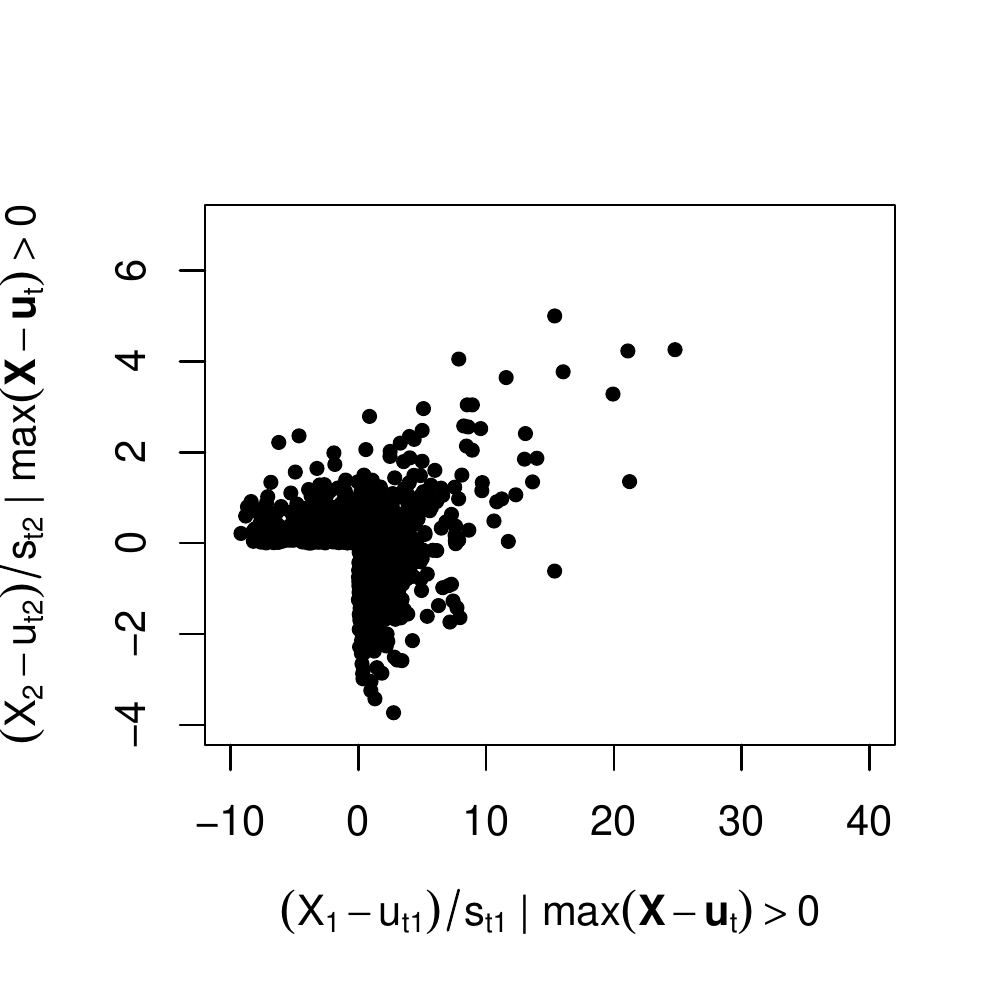}
 \caption{Illustration of (T2). Left panel: points from a two-dimensional multivariate GP distribution with parameters $\bsigma=(2,0.5)$ and $\bgamma=(0.2,0.1)$. Centre: black points denote exceedances of the threshold $\bm{u}_t = \bsigma(t^\bgamma-1)/\bgamma$, for $t=5$. Right: excesses of $\bm{u}_t$ rescaled by $\bm{s}_t=t^\bgamma$ have the same distribution as points in the left panel, but are five times fewer.  In particular extremes  in right plot hence are smaller. }
 \label{fig:stability}
\end{figure}

If (T1) holds we say that $F$ belongs to the (threshold) domain of attraction of $H$.  In contrast to the limit in (M1), different threshold functions can lead to limits which are not location-scale transformations of one another. A cdf $F$ is in a domain of attraction for maxima if and only if it is in a threshold domain of attraction.

 The GP distribution $H$ is supported by the set
\[
  [\bEta, \bomega] \setminus [\bEta, \bzero]
  = \{ \bx \in \reals^d \;:\; \text{$\eta_j \leq x_j \leq \omega_j$ for all $j$, and $x_j > 0$ for some $j$} \}.
\]
It may assign positive mass to the hyperplanes $\{ \by : y_j = \eta_j \}$, even if $\eta_j = -\infty$; see Example~\ref{ex:independentFrechet}  below.

For a non-empty subset $J$ of $\{1, \ldots, d\}$, let $H_J$ denote the corresponding $\abs{J}$-variate marginal distribution of $H$. Further, let $H_J^+$ denote $H_J$ conditioned to have at least one positive component; this presupposes that $\sigma_j > 0$ for some $j \in J$, where $\sigma_j = \alpha_j - \mu_j \gamma_j$ as before. By Theorem~\ref{prop:GPproperties}(i) below, if $\bsigma > \bzero$, then $H_j^+:=H_{\{j\}}^+$, the $j$-th marginal distribution of $H$, conditioned to be positive, has cdf
\begin{equation}
\label{eq:gpmarginspos}
  H_j^+(x) = 1 - \left( 1 + \gamma_j \tfrac{x}{\sigma_j} \right)^{-1/\gamma_j},
  \qquad \text{for $x \ge 0$ such that $\sigma_j + \gamma_j x > 0$}.
\end{equation}
This proves the intuitively appealing result that $H_j^+$ is a one-dimensional GP distribution, and shows that $\bsigma, \bgamma$ and then also $\bEta$ are determined by the values of $ H( \bx )$ for $ \bx \geq \bzero$.

If $J$ is a non-empty subset of $\{1, \ldots, d\}$ and $\bx \in [-\infty, \infty]^J$, then $\bar \bx \in [-\infty, \infty]^d$ is defined by $\bar x_j = x_j$ if $j \in J$ and $\bar x_j = \infty$ if $j \not\in J$.
Thus, if $\bX \sim H$, then the marginal distribution, $H_J$, of $(X_j : j \in J)$ is given by $H_J(\bx)=H(\bar\bx)$ for $\bx \in [-\infty,\infty]^J$, and if $H_J(\bzero)<1$ then
\begin{equation}\label{eq:condposmargins}
  H^+_J(\bx)
  = \frac{H_J(\bx)-H_J(\bx \wedge \bzero)}{\bar{H}_J(\bzero)}
\end{equation}
is the conditional distribution of $(X_j : j\in J)$ given that $\max_{j \in J} X_j > 0$, see~\eqref{eq:conditioning} above.  Recall that  $G$ and $H$  are said to be associated, $H \leftrightarrow G$, if they satisfy \eqref{eq:MGPD}.

\begin{theorem}
\label{prop:GPproperties}
Let $G$ be a GEV with margins \eqref{eq:gmarginals} and suppose $H \leftrightarrow G$.
\begin{compactenum}[(i)]
\item
Let $J \subset \{1, \ldots, d\}$. If $H_J(\bzero)<1$ then $H^+_J$ is a GP cdf too, with $H^+_J \leftrightarrow G_J$, and if $\sigma_j > 0$ then \eqref{eq:gpmarginspos} holds. Further, $H_J$ is a GP distribution if and only if $H_J(\bzero)=0$.
\item
A scale transformation of H is also a GP distribution.
\item
Let $\bX \sim H$. If $\bsigma > \bzero$ and $\bu \geq \bzero$ with $H(\bu) < \bone$, then the conditional distribution  of $\bX- \bu$ given that $\bX \nleq \bu$ is a GP distribution with the same shape parameter $\bgamma$ and $\bsigma$ replaced by $\bsigma + \bgamma \bu$.
\item
If $\{H_n\}$ is a sequence of GP distributions with all components of the vectors $\bsigma_n$ bounded away from $0$ and if $H_n \dto \tilde{H}$ then $\tilde{H}$ is a GP distribution too.
\item
A finite or infinite mixture of GP distributions with the same $\bsigma$ and $\bgamma$ is a GP distribution.
\item
We have $H \leftrightarrow G^t$ for all $t>0$. Conversely, if $H \leftrightarrow G_*$ for some $G_*$ and if $\bsigma > \bzero$ then $G_*= G^{t_1}$ for some $t_1 >0$.
\item
If $G(\bzero)=e^{-v}$ then $G(\bx) = \exp\{-v \bar{H}(\bx)\}$, $\bx \geq \bzero$, and if $\bsigma > \bzero$ this determines $G$.
\item
If $\bsigma > \bzero$, the scaling and threshold functions in the (T2) characterization of GP distributions may be taken as $\bs_t = t^{\bgamma}$ and $\bu_t = \bsigma (t^{\bgamma} - \bone)/\bgamma$, for $t\geq 1$.
\item
The parameters $\bgamma$ and $\bsigma$ are identifiable from $H$.
\end{compactenum}
\end{theorem}

In words, Theorem~\ref{prop:GPproperties}(i) says that conditional margins of GP distributions are GP, but that marginal distributions of GP distribution are typically not GP. For instance, if $H$ is a two-dimensional GP cdf, then $H_1^+$ is a one-dimensional GP cdf (given by \eqref{eq:gpmarginspos}), but typically $H_1$ is not. Intuitively, the reason is that the conditioning event implicit in $H_1(x)$ also includes the possibility that it is the second component, rather than the first one, that exceeds its threshold. Theorem~\ref{prop:GPproperties} (ii)$-$(v) also establish closure properties of the class of GP distributions. By (vi) and (vii) a GEV distribution specifies the associated GP distribution and conversely a GP distribution specifies a curve of associated GEV distributions in the space of distribution functions. Regarding (vii), note that a GEV distribution $G$ such that $0 < G_j(0) < 1$ for all $j$ is determined by its values for $\bx \ge \bzero$ (proof in the Appendix). Finally, (viii) identifies the affine transformations which leave $H$ unchanged, and (ix) establishes identifiability of the marginal parameters.

\begin{proof}
(i) Let $\bar\bzero$ denote $\bar\bx$ for the special case when $\bx = \bzero \in (-\infty, \infty)^J$ and let $G_J(\bx)=G(\bar\bx)$ be the marginal distribution of $G$.  Clearly $\bar\bx\wedge \bzero =\overline{ (\bx \wedge \bzero)} \wedge \bzero$ and hence, for $\bx > \bEta$,
\begin{eqnarray*}
  \lefteqn{
    H_J(\bx)-H_J(\bx \wedge \bzero)
  } \\
  &=& \dfrac{1}{\log G( \bzero )} \log \left( \dfrac{ G( \bar\bx \wedge \bzero ) }{ G( \bar\bx) }  \right) - \dfrac{1}{\log G( \bzero )} \log \left( \dfrac{ G( \overline{(\bx \wedge \bzero)} \wedge \bzero ) }{ G(\overline{\bx \wedge \bzero}) }  \right) \\
         &=& \dfrac{1}{\log G( \bzero )} \log \left( \dfrac{ G_J( \bx \wedge \bzero ) }{ G_J( \bx ) } \right)
\end{eqnarray*}
  and
\[
 \bar{H}_J(\bzero) = 1- \dfrac{1}{\log G( \bzero )} \log \left( \dfrac{ G( \bar\bzero \wedge \bzero) }{ G( \bar\bzero) }\right)
= \dfrac{\log   G_J( \bzero)}{\log G( \bzero )},
\]
so that
$$
H_J^+(\bx) = \dfrac{1}{\log G_J( \bzero )} \log \left( \dfrac{ G_J( \bx \wedge \bzero ) }{ G_J( \bx ) } \right).
$$
Inserting \eqref{eq:gmarginals} into the equation above for $J = \{j\}$  together with straightforward calculation proves \eqref{eq:gpmarginspos}, and hence completes the proof of the first  assertion.

If $H_J(\bzero)=0$,  then $H_J=H^+_J$ and it follows from the first assertion  that $H_J$ is a GP distribution function. Further, GP distributions are supported by $\{\by; \by \nleq \bzero\}$ and hence if $H_J(\bzero)> 0$ then $H_J$ is not a GP cdf. This proves the second assertion.

(ii) If $G$ is a GEV cdf then, for $\bs > \bzero$, the map $\bx \mapsto G(\bx/\bs)$ is a GEV cdf too, and the result then follows from
\[
  H(\bx/\bs)
  = \frac{1}{\log G(\bzero)} \log \left( \frac{G((\bx /\bs) \wedge \bzero)}{\log G(\bx/\bs)}\right)
  = \frac{1}{\log G(\bzero/\bs)} \log \left( \frac{G((\bx \wedge \bzero)/\bs)}{G(\bx/\bs)}\right).
\]

(iii)   Proceeding as in the proof of (i), but in the first step instead using that  $(\bx  +\bu)\wedge \bzero  = (\bx \wedge \bzero   +\bu) \wedge \bzero$,  shows that the conditional distribution of $\bX - \bu$ given that $\bX \nleq \bu$ is
\begin{equation*}
  \frac{H(\bx + \bu) - H(\bx \wedge \bzero + \bu)}{\bar{H}(\bu)}
  = \frac{1}{\log G(\bu)} \log \left( \frac{G(\bx  \wedge \bzero  +\bu)}{G(\bx + \bu)}\right).
\end{equation*}
The map $\bx \mapsto \tilde{G}(\bx) = G(\bx + \bu)$ is also a GEV cdf, but with the vector $\bsigma = \balpha - \bgamma \bmu$ replaced by $\tilde{\bsigma} = \balpha - \bgamma (\bmu - \bu) = \bsigma + \bgamma \bu$.

(iv) Convergence in distribution in $\reals^d$ implies convergence of the marginal distributions, and using standard converging subsequence arguments it follows from marginal convergence that there exist $\bsigma > \bzero$ and $\bgamma$ such that $\bsigma_n \to \bsigma$ and $\bgamma_n \to \bgamma$.  Define $\bu_{n,t}$ and $\bs_{n,t}$ from $H_n$ as in Equation~\ref{eq:equality} (vi). Then, since $H_n$ is a GP cdf we have, using first (T2) and (viii), and then the continuous mapping theorem, that
  \begin{align*}
    H_n(\bx)
    &= \frac{H_n(\bx/\bs_{n,t} + \bu_{n,t}) - H_n((\bx/\bs_{n,t}) \wedge \bzero + \bu_{n,t})}{\bar{H}_n(\bu_{n,t})}  \\
    &\stackrel{d}{\to} \frac{\tilde{H}(\bx/\bs_t + \bu_t) - \tilde{H}((\bx/\bs_t) \wedge \bzero + \bu_t)}{\bar{\tilde{H}}(\bu_t)},
    \qquad \text{as}\ n \to \infty.
  \end{align*}
Since $H_n \stackrel{d}{\to} \tilde{H}$ it follows that $\tilde{H}$ satisfies (T2) and hence is a GP cdf.

(v)  We only prove that a mixture of two GP cdf-s with the same $\bsigma$ and $\bgamma$ is a GP cdf too, using Theorem~\ref{thm:representation} below (the proof of that theorem does not use the result we are proving now). The proof for arbitrary finite mixtures is the same, and the result for infinite mixtures then follows by taking limits of finite mixtures and using (iv). Let $H_1$ and $H_2$ be GP cdf-s with the same marginal parameters $\bsigma, \bgamma$ and let $p \in (0,1)$. By Theorem~\ref{thm:representation} and Equation~\eqref{eq:MGP:R} there exists cdf-s $F_i=F_{Ri}$ such that $H_i(\bx) = c_i \int_0^\infty \{F_{i}(t^{\bgamma}(\bx+\frac{ \bsigma}{ \bgamma})) - F_{i}(t^{\bgamma}(\bx \wedge \bzero + \frac{ \bsigma}{ \bgamma}))\} \, \diff t $ with $c_i=1/\int_0^\infty \bar F_i(t^{\bgamma}{\frac{ \bsigma}{ \bgamma}}) \, \diff t$, and with the convention that if $\gamma_i=0$ then $t^{\gamma_i}(x_i+\frac{ \sigma_i}{ \gamma_i})$ is interpreted to mean $x_i+\sigma_i \log t$. Writing $F=\frac{pc_1}{pc_1+(1-p)c_2}F_1 + \frac{(1-p)c_2}{pc_1+(1-p)c_2}F_2$ it follows that
\begin{eqnarray*}
  \tilde{H}(\bx)&:=&pH_1(\bx)+(1-p)H_2(\bx)\\
  &=&
  [pc_1+(1-p)c_2]
  \int_0^\infty
    \left\{
      F \left( t^{\bgamma} (\bx + \tfrac{ \bsigma}{ \bgamma}) \right) -
      F \left( t^{\bgamma} (\bx \wedge \bzero + \tfrac{ \bsigma}{ \bgamma}) \right)
    \right\} \,
  \diff t.
\end{eqnarray*}
Straightforward calculation shows that $pc_1+(1-p)c_2=1/\int_0^\infty \bar F(t^{\bgamma}{\frac{ \bsigma}{ \bgamma}}) \, \diff t$ so that $\tilde{H}(\bx)$ satisfies Equation~\eqref{eq:MGP:R} and hence is a GP cdf.

(vi) The first assertion follows from \eqref{eq:MGPD}. Choose $t_1$ so that $- \log G(\bzero)^{t_1} = - \log G_*(\bzero)$.  Then $H \leftrightarrow G$ and $H \leftrightarrow G_*$ imply together that
  \begin{equation*}
  \frac{G(\bx \wedge \bzero)^{t_1}}{G(\bx)^{t_1} } = \frac{G_*(\bx \wedge \bzero)}{G_*(\bx) },
  \end{equation*}
and in particular, that $G(\bx)^{t_1} = G_*(\bx)$ for $\bx\geq \bzero$. Since a GEV cdf with $ \bsigma > \bzero$ is determined by its values for $\bx \geq \bzero$, see the Appendix, this completes the proof.

(vii) The first part follows from \eqref{eq:MGPD}, and that this determines $G$ again follows from the appendix.

(viii) By the proofs of (iii) and (vi) the conditional distribution of $\bs_t^{-1}(\bX - \bu_t)$ given that $\bX \nleq \bu_t$ is associated with $G(\bs_t \bx + \bu_t ) = G(\bx)^{1/t}$,  and the result follows from (vii).

(ix) For $t > 0$, let $(\gamma_j(t), \mu_j(t), \alpha_j(t)) \in \reals \times \reals \times (0, \infty)$ be the parameter vector of $G_j^t$, and let $\sigma_j(t) = \alpha_j(t) - \gamma_j(t) \, \mu_j(t)$. By assertion (vi), the GPD $H$ determines the curve of GEV-s $G^t$ for $t > 0$. It suffices to show that $\gamma_j(t)$ and $\sigma_j(t)$ do not depend on $t$. But this follows by straightforward calculations from the max-stability property $G(\bx)^t = G( \ba_t^{-1}( \bx - \bb_t) )$ with $\ba_t$ and $\bb_t$ as in \eqref{eq:tparameters}.
\end{proof}

Example \ref{ex:independentFrechet}  below exhibits two-dimensional GP distributions with positive mass on certain lines, and the first part of Example~\ref{ex:comonotoneFrechet}  provides a cdf where the second assertion in (i) of Theorem~\ref{prop:GPproperties} comes into play. In contrast to  scale transformations, it seems likely that if $\bsigma > \bzero$ then a non-trivial location transformation of a GP cdf never is a GP cdf. The second part of Example~\ref{ex:comonotoneFrechet} shows one of the exceptional cases where the support of one of the components is contained in $(-\infty, 0)$ and where  a location transformation of a GP distribution does give another GP distribution.

\begin{example}
\label{ex:independentFrechet}
This example rectifies the one on pages 1726--1727 in \cite{ferreira2014}. Let $G(x, y) = \exp\{ - 1/(x+1) - 1/(y+1) \}$ for $(x, y) \in (-1, \infty)^2$, the distribution of two independent unit Fr\'echet random variables with lower endpoints $\alpha_1 = \alpha_2 = -1$. The corresponding multivariate generalized Pareto distribution is given by
\begin{equation}
\label{eq:independentFrechet}
  H(x, y) =
  \begin{cases}
    \frac{1}{2} \left( 1 - \frac{1}{x+1} + 1 - \frac{1}{y+1} \right)
    & \text{if $(x, y) \in [0, \infty)^2$,} \\[1ex]
    \frac{1}{2} \left( 1 - \frac{1}{x+1} \right)
    & \text{if $(x, y) \in [0, \infty) \times [-1, 0]$,} \\[1ex]
    \frac{1}{2} \left( 1 - \frac{1}{y+1} \right)
    & \text{if $(x, y) \in [-1, 0] \times [0, \infty)$,} \\[1ex]
    0
    & \text{otherwise.}
  \end{cases}
\end{equation}
We conclude that $H$ is the distribution function of the random vector $(X, Y)$ given by
\[
  (X, Y) =
  \begin{cases}
    (-1, T) & \text{with probability $1/2$,} \\
    (T, -1) & \text{with probability $1/2$,}
  \end{cases}
\]
where $T$ is generalized Pareto,  $\Pr(T \le t) = 1 - 1/(t+1)$ for $t \geq 0$. Hence $H$ is supported by the union of the two lines $\{-1\} \times (0, \infty)$ and $(0, \infty) \times \{-1\}$, see  Figure~\ref{fig:support}, left panel.

\begin{figure}
\begin{center}
\begin{tabular}{cc}
\includegraphics[width=0.3\textwidth]{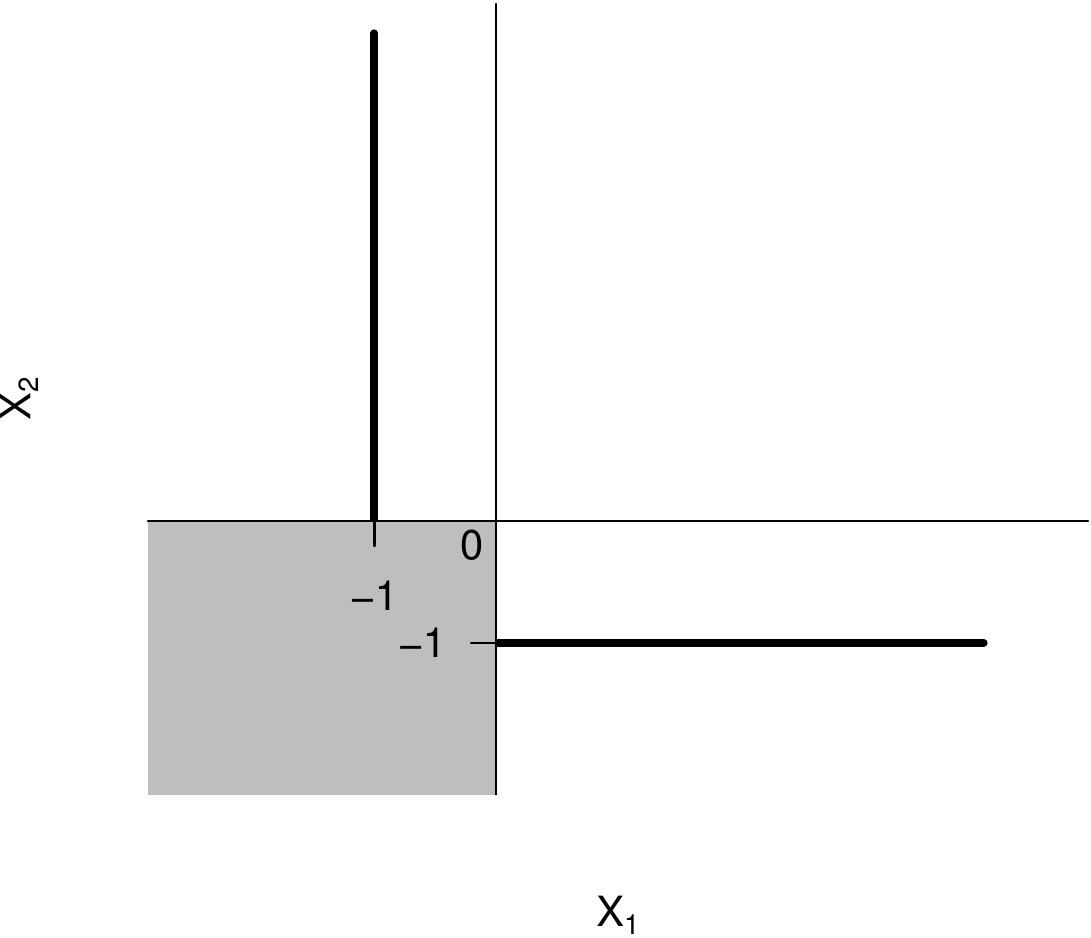}&
\includegraphics[width=0.3\textwidth]{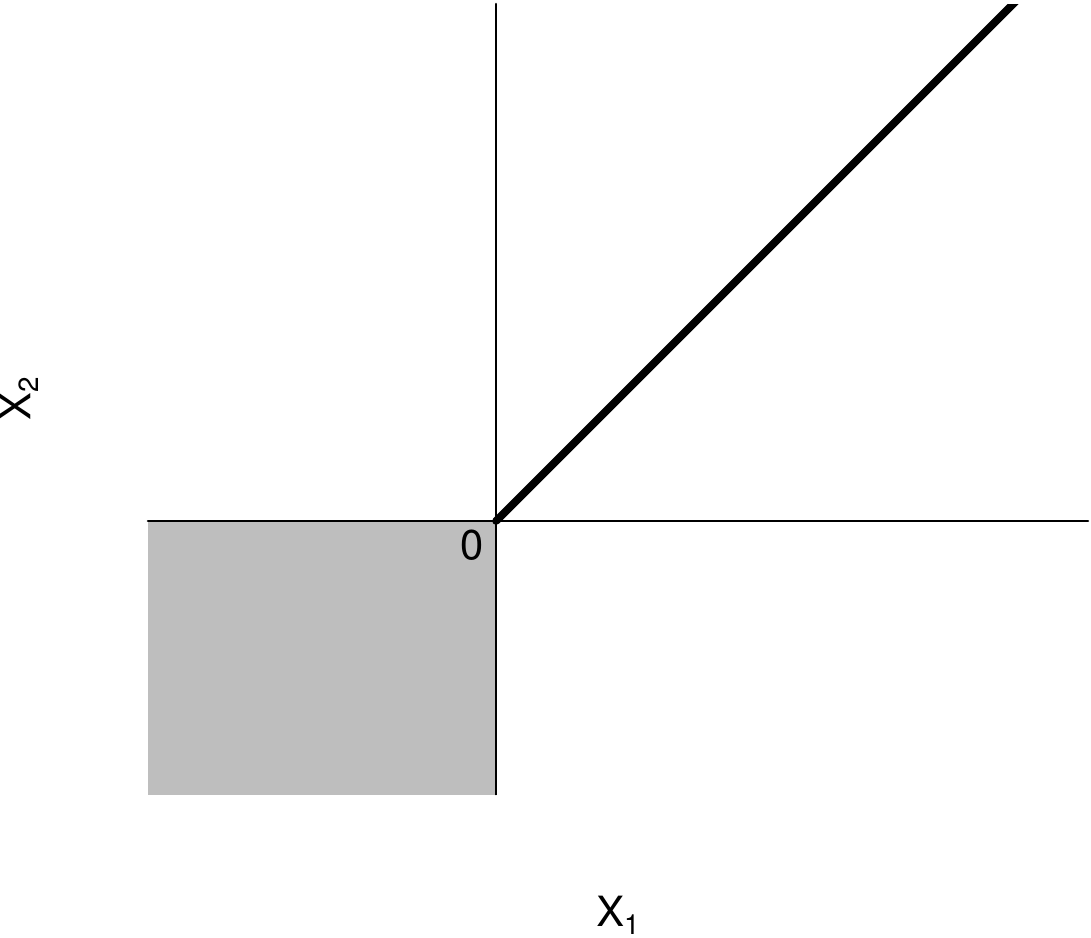}
\includegraphics[width=0.3\textwidth]{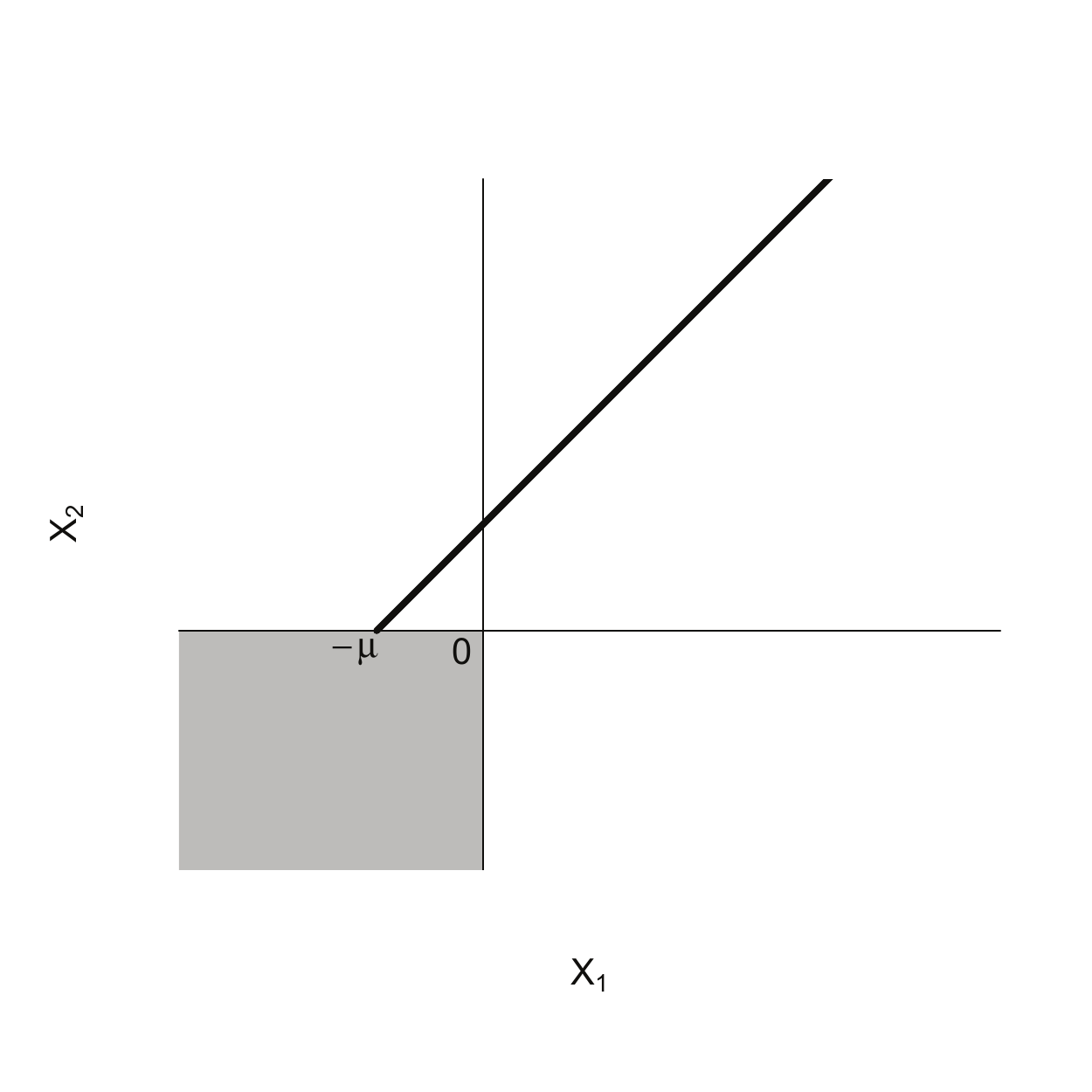}
\end{tabular}
\end{center}
\caption{\label{fig:support} \it Supports (solid lines) of the GP distributions $H$ in equations \eqref{eq:independentFrechet} (left), \eqref{eq:dependentFrechet} (middle) and \eqref{eq:locationshift} (right). }
\end{figure}

If we modify the example by choosing Gumbel rather than Fr\'echet margins, so that $G(x, y) = \exp( -e^{-x} - e^{-y} )$ for $(x, y) \in \reals^2$, then the GP cdf $H$ is the cdf of the vector
\[
  (X, Y) =
  \begin{cases}
    (-\infty, T) & \text{with probability $1/2$,} \\
    (T, -\infty) & \text{with probability $1/2$,}
  \end{cases}
\]
where $T$ is a unit exponential random variable, $\Pr(T \le t) = 1 - e^{-t}$ for $t \in [0, \infty)$. The support of $H$ is  the union of two lines $\{-\infty\} \times (0, \infty)$ and $(0, \infty) \times \{-\infty\}$ through $-\infty$.
\end{example}

\begin{example}
\label{ex:comonotoneFrechet}
Let $G(x, y) = \exp[ - 1 / \{ (x \wedge y) + 1 \} ]$ for $(x, y) \geq -\bone$, the cdf of  $(Z, Z)$ for $Z$   unit Fr\'echet  with lower endpoint $-1$. The corresponding GP cdf is
\begin{equation}
\label{eq:dependentFrechet}
  H(x, y) =
  \begin{cases}
    1 - \frac{1}{(x \wedge y) + 1}
      & \text{if $(x, y) \in [0, \infty)^2$,} \\[1ex]
    0
      & \text{otherwise.}
  \end{cases}
\end{equation}
We identify $H$ as the distribution of the random pair $(T, T)$, where $\Pr(T \le t) = 1 - 1/(t+1)$ for $t \in [0, \infty)$. The support of $H$ is the diagonal $\{ (t, t) : 0 < t < \infty \}$, see  Figure~\ref{fig:support}, middle panel. It follows that, e.g.,  $H_1(0)=0$ and hence in this example $H_1=H^+_1$.

As a variation of the example  let $G(x, y) = \exp[ - e^ {- x \wedge (y+\mu)} ]$  be the cdf of $(Z, Z -\mu)$, with $Z$ standard Gumbel and $\mu >0$. The corresponding GP cdf is
\begin{equation}\label{eq:locationshift}
  H(x,y)  = e^{-(x \wedge 0) \wedge (y \wedge 0 + \mu)} - e^{-x \wedge (y  + \mu)} = e^{-(x \wedge 0) \wedge (y + \mu)} - e^{-x \wedge (y  + \mu)},
\end{equation}
  and $H$ is the cdf of $(T, T- \mu)$ with $T$  standard exponential.
  Now, for $-\mu < \nu$ the location transformed cdf $H(x,y+\nu)$ equals $e^{-(x \wedge 0) \wedge (y + \nu +\mu)} - e^{-x \wedge (y  + \nu +\mu)}$, which is the same as in \eqref{eq:locationshift}, but with $\mu$ replaced by $\nu + \mu >0$. Hence also $H(x,y+\nu)$ is a GP cdf. The support of $H$ is shown in the right-hand panel of Figure \ref{fig:support}.
\end{example}

%
%
%

\section{Point processes of extreme episodes}
\label{sec:point processes}

The first step in our program for PoT inference is to specify a point process model for extreme episodes.  This model exhibits extreme episodes as a product process obtained by multiplying a random vector, the ``shape'' vector, with a random quantity, the ``intensity'' of the episode. (For $\gamma = 0$, the model instead is a sum.)  This is parallel to models commonly used for max-stable processes, see e.g.\ \citet{schlather2002}. In Section~\ref{sec:representation} below we obtain basic and physically interpretable representations of the GP distributions by conditioning the product process of  extreme episodes on threshold exceedance.

In this and subsequent sections we assume that $\bsigma > \bzero$. Let $\bX_1, \bX_2, \ldots$ be i.i.d.\ random vectors with cdf $F$ and marginal cdf-s $F_1, \ldots, F_d$, and let $\ba_n, \bb_n$ be as in \eqref{eq:DA}. Further, let $\bEta$ be the vector of lower endpoints of the limiting GEV distribution, see \eqref{eq:Gsupport} and the sentences right below it. We consider weak limits of the point processes
\[
  N_n=\sum_{i=1}^n \delta_{\ba_n^{-1}(\bX_i-\bb_n)},
\]
where $\delta_{\bx}$  denotes a point mass at $\bx$. Define
$I_j = [-\sigma_j/\gamma_j, \infty)$ or $[-\infty, \infty)$ or $[-\infty, -\sigma_j/\gamma_j)$ according to whether $\gamma_j > 0$ or $\gamma_j = 0$ or $\gamma_j < 0$, and set
\begin{equation*}
\bar{S}_{\bgamma}= I_1 \times \cdots \times I_d \; \;\; \mbox{and} \;\;\;
S_{\bgamma}= \bar{S}_{\bgamma} \setminus \{\bEta\}.
\end{equation*}
The limit point process is specified as follows:  Let $0 < T_1 < T_2 < \ldots$ be the points of a Poisson process on $[0, \infty)$ with unit intensity and let  $( \bR_i )_{i \ge 1}$ be independent copies of a random vector $\bR $ which satisfies Condition~\ref{cond:R} below. Further assume that the vectors $(\bR_i)_{i \ge 1}$  are independent of $(T_i)_{i \ge 1}$, and define the point process
\begin{equation}
\label{eq:PPV}
  P_r = \sum_{i \ge 1} \delta_{({\bR_i}/T_i^\bgamma-\bsigma/\bgamma)},
\end{equation}
where, by convention, $R_{i, j}/T_i^0-\sigma_j/0$ is interpreted to mean $R_{i,j} - \sigma_j \log T$. The condition on $\bR$ is as follows.

\begin{condition} \label{cond:R}
The components of the random vector $\bR$  satisfy $R_j \in [0, \infty) $ if $\gamma_j >0$,  $R_j \in [-\infty, \infty)$ if $\gamma_j =0$, and  $R_j \in [-\infty, 0)$ if $\gamma_j <0$, and furthermore $0 < \Exp[ \abs{R_j}^{1/\gamma_j} ] < \infty$ if $\gamma_j \neq0$ and $\Exp[\exp (R_j/\sigma_j)] < \infty$ if $\gamma_j=0$, for $j=1, \ldots, d$.
\end{condition}

Let $F_{\bR}$ be the cdf of $\bR$. For $\gamma_j \neq 0$, the moment restriction in Condition~\ref{cond:R} can be seen to be equivalent to requiring that $0<\int_{0}^{\infty} \Pr(R_j>t^{\gamma_j} x_j) \, \diff t < \infty$, if $x_j\in (0, \infty)$ and $\gamma_j>0$ or if $x_j\in (-\infty, 0)$ and $\gamma_j<0$.
For $\gamma_j = 0$, the moment condition is instead equivalent to $0<\int_{0}^{\infty} \Pr(R_j>\sigma_j \log t + x_j) \, \diff t < \infty$, for $x_j \in (-\infty , \infty)$.  For example, if $\gamma_j < 0$, then $\int_{0}^{\infty}P(R_j > t^{\gamma_j} x_j)  \, \diff t = \int_{0}^{\infty} \Pr(\abs{R_j}^{1/\gamma_j}>t) \, \diff t \, \abs{x_j}^{-1/\gamma_j} = \Exp(\abs{R_j}^{1/\gamma_j}) \, \abs{x_j}^{-1/\gamma_j}$. Since $\Pr(R_j>t^{\gamma_j} x_j) \leq \bar F_{\bR}(t^{\bgamma}\bx) \leq \sum_{i=1}^{d} \Pr(R_i>t^{\gamma_i} x_i)$, it in turn follows that the moment condition implies that $0 < \int_{0}^{\infty}\bar{F}_{\bR}(t^{\bgamma}(\bx + \bsigma/\bgamma)) \, \diff t < \infty$,  if the components $x_j$ of $\bx$ are as above, and  where we have used the convention that $t^0 (x_j + \sigma_j/0)$  should be replaced by  $\sigma_j \log t + x_j$.

\begin{theorem}
\label{thm:PPconv}
Suppose $F$ satisfies \eqref{eq:DA}. Then, for some $\bR$ which satisfies Condition~\ref{cond:R},
\begin{equation}
\label{eq:PPconv}
 N_n \stackrel{d}{\to} P_r \; \; \text{on} \; \; S_{\bgamma}, \; \; \text{as} \; \; n \to \infty.
\end{equation}
Conversely, for any $P_r$ given by \eqref{eq:PPV} there exist a GEV cdf $G$  and  $\ba_n > \bzero$ and $\bb_n$, with $0 < G_j(b_{n,j}) < 1$ and $G_j(b_{n,j}) \to 1$ for  $j = 1, \ldots, d$, such that \eqref{eq:PPconv} holds for $F=G$.
\end{theorem}

\begin{proof}

Let $\bY \sim G$ and define $\bY^*$ by $Y^*_j = (1+\frac{\gamma_j}{\alpha_j}(Y_j-\mu_j))^{1/\gamma_j} $ if $\gamma_j \ne 0$ and $ Y^*_j = \exp\{(Y_j-\mu_j)/\alpha_j\}$ if $\gamma_j = 0$, for $\bmu, \balpha, \bgamma$  given by \eqref{eq:gmarginals} so that the  marginal cdf-s of $\bY^*$ are standard Fr\'echet.   It follows as in Theorem~5 of \citet{penrose1992} (see also \cite{dehaanferreira2006} and \citet{schlather2002}) that there exists a random vector $\bR^* \in [0, \infty)^d$ with $\Exp(R_j^*)< \infty$  such that $\bY^*$ has the same distribution as $\sup_{i \ge 1} \bR_i^*/T_i$ where the random vectors $\bR^*_i$ are i.i.d.\ copies of $\bR^*$ and independent of the unit rate Poisson process $(T_i)_{i \ge 1}$. Reversing the transformation which led from $\bY$ to $\bY^*$, it follows that $\bY$ has the same distribution as $\sup_{i \ge 1} (\frac{\balpha}{\bgamma}(\bR_i^*)^{\bgamma}/T_i^{\bgamma} -  \tfrac{\bsigma}{\bgamma})$. Setting $\bR = \frac{\balpha}{\bgamma}(\bR^*)^{\bgamma}$ it follow that $\bY$ has the same distribution as $\sup_{i \ge 1} (\bR_i/T_i^{\bgamma} -  \tfrac{\bsigma}{\bgamma})$ with the $\bR_i$ satisfying Condition~\ref{cond:R}, and where throughout expressions should be interpreted as specified after Equation~\ref{eq:PPV} for $\gamma_j =0$.

For $\nu$ the intensity measure of $P_r$, we have $G(\bx)=\exp\{-\nu(\{\by : \by \nleq \bx\})\}$. By standard reasoning, convergence in distribution of $\ba_n^{-1} (\bigvee_{i=1}^n \bX_i - \bb_n)$ is equivalent to $n \Pr[\ba_n^{-1} (\bX_1 - \bb_n) \nleq \bx] \to - \log G(\bx) = \nu(\{\by : \by \nleq \bx\})$, which implies that $n \Pr[\ba_n^{-1} (\bX_1 - \bb_n) \in \point]$ converges vaguely to $\nu(\point)$ on $S_{\bgamma}$. By Theorem~5.3 of \citet{resnick2007} this proves \eqref{eq:PPconv}.

Conversely,  given $P_r$, define the cdf $G$ by $G(\bx) = \Pr [ \sup_{i \ge 1} (\bR_i/T_i^\bgamma -  \tfrac{\bsigma}{\bgamma}) \leq \bx ]$.
Straightforward calculation as  in \citet[Theorem~2]{schlather2002}, show that $G$ is max-stable. Hence there exist sequences $\ba_n$ and $\bb_n$ with the stated properties such that for independent random vectors $\bX_1, \bX_2, \ldots$ with common distribution $G$, the distribution of $\ba_n^{-1}(\bigvee_{i=1}^n \bX_i-\bb_n)$ is equal to $G$ too. By the first part of the proof, this proves \eqref{eq:PPconv}.
\end{proof}

In the proof of Theorem \ref{thm:PPconv} we obtained part of the following result, which we record here for completeness.

\begin{corollary}
\label{cor:MGEV}
Suppose $\bR$ satisfies Condition~\ref{cond:R}. Then $G(\bx) = \Pr[\sup_{i \ge 1} (\bR_i/T_i^\bgamma -  \tfrac{\bsigma}{\bgamma})\leq \bx]$ is a GEV cdf and
\begin{equation}\label{eq:GEV-int}
  G(\bx) =
  \exp\left\{ - \int_0^\infty \bar{F}_{\bR} \left( t^{\bgamma} \left( \bx + \tfrac{\bsigma}{\bgamma} \right) \right) \, \diff t\right\}
  \; \;\; \text{for} \;\;\; \bx \in S_{\bgamma},
\end{equation}
and to any GEV cdf there exists an $\bR$ which satisfies this equation. Here we  use the convention that $t^0 (x_j + \sigma_j/0)$  should be replaced by  $\sigma_j \log t + x_j$.
\end{corollary}

\begin{proof}
Writing $\nu$ for the intensity measure of $P_r$ we have $\nu(\{\by; \by \nleq \bx \}) = \int_0^\infty \bar{F}_{\bR} ( t^{\bgamma}( \bx + \tfrac{\bsigma}{\bgamma} )) \, \diff t$. The right-hand side of \eqref{eq:GEV-int} is therefore equal to the probability that $P_r$ has no points in the set $\{\by; \by \nleq \bx \}$. The result now follows from the proof of Theorem~\ref{thm:PPconv}.
\end{proof}


\section{Representations of multivariate GP distributions}
\label{sec:representation}

This section contains the second step in the program for PoT inference. We show how conditioning on  threshold exceedances in the point process \eqref{eq:PPV} gives four widely useful representations of the class of multivariate GP distributions. The first representation, $(R)$ is on the real scale and corresponds to the point process $P_r$ in \eqref{eq:PPV} with points obtained as products of  shape vectors and intensity variables. In the second representation, $(U)$, the basic model is constructed on a standard scale and then transformed to the real scale. The third representation, $(S)$, is equivalent to the spectral representation in \cite{ferreira2014}. A fourth representation, $(T)$, which is a variation of the $(S)$ representation, is introduced in Section~\ref{sec:densities}.

In the literature the standard scale is  chosen as one of the following: Pareto scale, $\bgamma=\bone$, uniform scale,  $\bgamma=-\bone$,  or exponential scale, $\bgamma=\bzero$. Here we choose the  $\bgamma=\bzero$ scale  because of the simple additive structure it leads to. For all four representations, it is straightforward to switch from one scale to another one.

To understand the GP representation $(R)$  we first approximate $P_r$  by a truncated point process $\bar{P}_r$ where  $\{T_i\}$  is  replaced by a unit rate Poisson process $\{\bar{T}_i\}$ on a bounded interval $[0, K]$. Recalling the  representation of $\{\bar{T}_i\}$ as a Poisson distributed number of $\text{Unif}\, [0, K]$ variables,   $\bar{P}_r$ consists of a Poisson number of vectors $\bR/\bar{T}^{\bgamma} - \tfrac{\bsigma}{\bgamma}$ with $\bar{T} \sim \text{Unif} \, [0, K]$. Thus, for large $n$, a point $\ba_n^{-1}(\bX - \bb_n)$ in $N_n$ has approximately the same distribution as $\bR/\bar{T}^{\bgamma} - \tfrac{\bsigma}{\bgamma}$. Hence, by \eqref{eq:conditioning},
\begin{eqnarray}
\label{eq:smgpdappr}
  \Pr[\ba_n^{-1}(\bX - \bb_n) \leq \bx \mid \bX - \bb_n \nleq \bzero]
  &\approx&
  \Pr[\bR/\bar{T}^{\bgamma} - \tfrac{\bsigma}{\bgamma} \leq \bx \mid \bR/\bar{T}^{\bgamma} - \tfrac{\bsigma}{\bgamma} \nleq \bzero] \\
  &=&
  \frac{\frac{1}{K}\int_0^K \{F_{\bR}(t^{\bgamma}(\bx+\tfrac{\bsigma}{\bgamma})) -F_{\bR}(t^{\bgamma}(\bx \wedge \mathbf{0} +\tfrac{\bsigma}{\bgamma}))\} \, \diff t}{\frac{1}{K}\int_0^K \bar F_{\bR}(t^{\bgamma}\tfrac{\bsigma}{\bgamma}) \, \diff t} \nonumber \\
  &=&
  \frac{\int_0^K \{F_{\bR}(t^{\bgamma}(\bx+\tfrac{\bsigma}{\bgamma})) -F_{\bR}(t^{\bgamma}(\bx \wedge \bzero +\tfrac{\bsigma}{\bgamma}))\} \, \diff t}{\int_0^K \bar F_{\bR}(t^{\bgamma}\tfrac{\bsigma}{\bgamma})\, \diff t} \nonumber.
\end{eqnarray}
By the assumptions in  Condition~\ref{cond:R}, the limit as $K \to  \infty$ of this expression,
\begin{equation}
\label{eq:MGP:R}
H_R(\bx) =\frac{\int_0^\infty \{F_{\bR}(t^{\bgamma}(\bx+\frac{ \bsigma}{ \bgamma})) - F_{\bR}(t^{\bgamma}(\bx \wedge \bzero + \frac{ \bsigma}{ \bgamma}))\} \, \diff t}{\int_0^\infty \bar F_{\bR}(t^{\bgamma}{\frac{ \bsigma}{ \bgamma}}) \, \diff t},
\end{equation}
 exists (cf the discussion just before Theorem \ref{thm:PPconv}), and it is also immediate that $H_R(\binfty) =1$, so that $H_R$ is a cdf on $[-\binfty, \binfty)^d$. If $\gamma_i = 0$ then $t^{\gamma_i}(x_i +\frac{ \sigma_i}{ \gamma_i})$ should be interpreted to mean $x_i + \sigma_i  \log t$.
We write $\GP_R$($\bsigma, \bgamma, F_{\bR}$) for the cdf \eqref{eq:MGP:R} and call it the $(R)$ representation.
Theorem \ref{thm:representation} shows that the class of such cdf-s is the same as the class of all GP cdf-s with $\bsigma > \bzero$.

Heuristically, for simplicity assuming that $\gamma_j \neq 0, j=1, \ldots d$, the calculations above proceed by equating $\ba_n^{-1}(\bX - \bb_n)$ with $\bR/\bar{T}^{\bgamma} - \tfrac{\bsigma}{\bgamma}$ so that extremes of $\bX$ asymptotically have  the form $\ba_n \bR/T^\bgamma + \bb_n -\ba_n \tfrac{\bsigma}{\bgamma}$. Setting $\bb= \bb_n - \ba_n  \tfrac{\bsigma}{\bgamma}$ and noting that $\bR$ satisfies Condition~\ref{cond:R} if and only if  $\ba_n \bR$ satisfies Condition~\ref{cond:R}, the intuition is that, asymptotically, extremes of $\bX$ have the form
\begin{equation}\label{eq:point}
\bX^\infty=\bR/T^\bgamma + \bb,
\end{equation}
for some constant $\bb$ and a random vector $\bR$ which satisfies Condition~\ref{cond:R}. The interpretation is that the vector  $\bR$ is the shape of the extreme episode, say a storm, and that $T^{-\bgamma}$ is the intensity of the storm. Here $T$ represents a pseudo-random variable with an improper uniform distribution on $(0,\infty)$. Although such a $\bX^\infty$ therefore does not have a proper distribution, one can verify that the cdf $H_R$ in  \eqref{eq:MGP:R} is derived as if it were the conditional distribution of $\bX^\infty -\bu$, given that $\bX^\infty \nleq \bu$, for $\tfrac{\bsigma}{\bgamma} =  \bu - \bb$, i.e., as the formal conditional distribution of $\bR/T^\bgamma - \tfrac{\bsigma}{\bgamma}$  given that  $\bR/T^\bgamma - \tfrac{\bsigma}{\bgamma} \nleq \bzero$. In  statistical application one would  assume that $\bu$ is large enough to make it possible to use the cdf $H_R$ as a model for threshold excesses.  The parameters of $\bR$ and the parameters $\bsigma$ and $\bgamma$ are then estimated from the observed threshold excesses. The heuristic interpretation in case one or more of the $\gamma_j$ equals $0$ is the same, one only has to write $R_j -\sigma_j \log T$  instead of $R_j/T^0 - \sigma_j/0$.

Figure~\ref{fig:PPtoGP} illustrates how the multivariate GP is derived from the Poisson process representation. Each realization of the Poisson process~\eqref{eq:PPV} yields a small, Poisson distributed, number of points in the region $\{\bx:\bx\nleq \bzero\}$. The expected number of such points is $\Exp[\bigvee_{j=1}^d(\gamma_jR_j/\sigma_j)^{1/\gamma_j}]$, where if $\gamma_j=0$ the component is to be interpreted as $e^{R_j/\sigma_j}$, and thus depends on the distribution of $\bR$ and the parameters $\bsigma,\bgamma$.

\begin{figure}
 \includegraphics[width=0.3\textwidth]{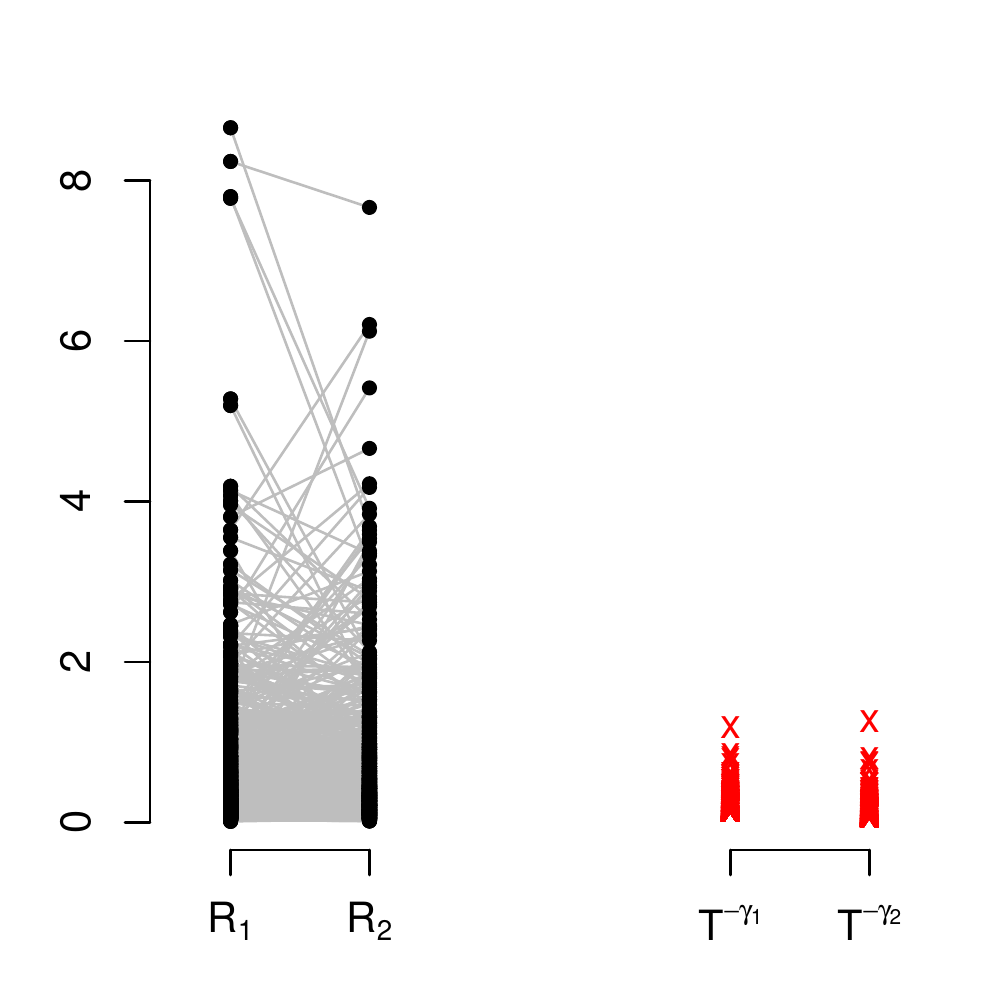}
 \includegraphics[width=0.3\textwidth]{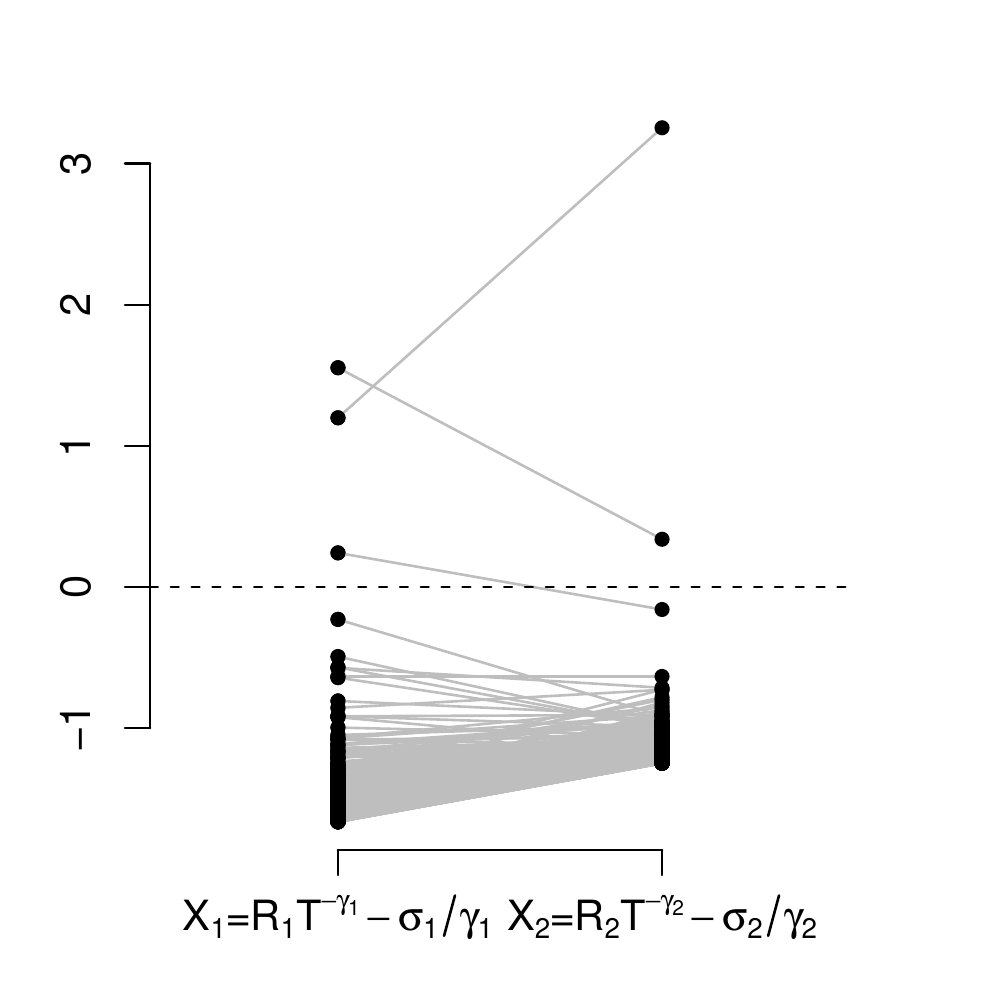}
 \includegraphics[width=0.3\textwidth]{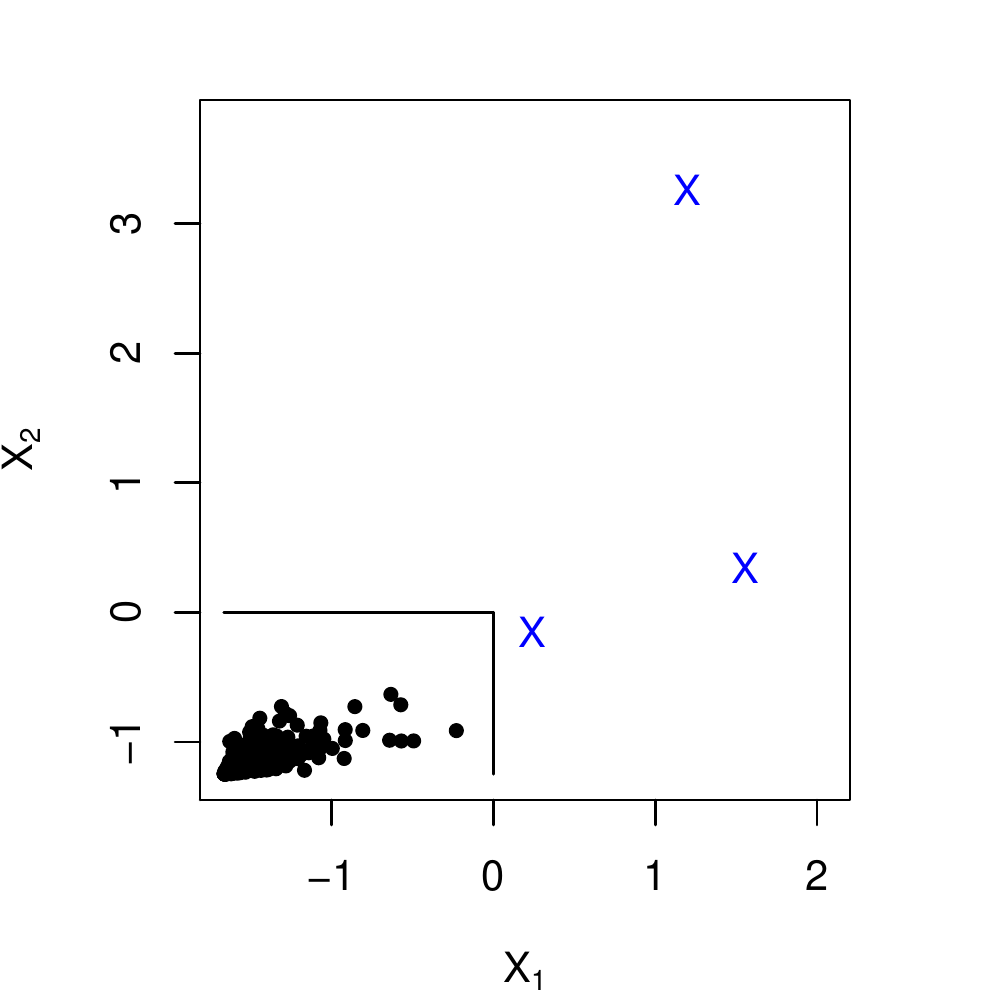}
 \caption{Deriving the GP from the Poisson process representation. Left: two-dimensional illustrations of ``shape vectors'' $\bR$ and the 1000 largest ``intensities'' $T^{-\bgamma}$ for $\bgamma=(0.3,0.4)$. Centre: points $\bX=\bR/T^\bgamma-\bsigma/\bgamma$, where $\bsigma=(0.5,0.5)$ against index, with a horizontal line at zero. Right: points $X_2$ versus $X_1$ with exceedances of zero in at least one coordinate highlighted.}
 \label{fig:PPtoGP}
\end{figure}

Defining $\bU$ by $ \tfrac{\bsigma}{\bgamma}e^{\bgamma \bU} = \bR$, where we use the convention that if $\gamma_j=0$ then the $j$-th component is given by  $\sigma_j U_j =R_j$,  we can write \eqref{eq:MGP:R} as
\begin{equation}
\label{eq:MGP:RS}
  H_U( \bx ) =
  \frac%
  {\int_0^\infty \, \bigl\{ F_{\bU} \bigl(\frac{1}{\bgamma}  \log(\frac{\bgamma }{\bsigma}\bx   + 1) + \log t \bigr) - F_{\bU} \bigl( \frac{1}{\bgamma}  \log(\frac{\bgamma }{\bsigma}(\bx\wedge \bzero) + 1) +  \log t \bigr) \bigr\} \, \diff t}%
  {\int_0^\infty \,  \bar{F}_{\bU}( \log t)  \, \diff t},
\end{equation}
for $\bx$ such that $\frac{\bgamma }{\bsigma}\bx   + \bone > \bzero$, and where $F_{\bU}$ is the cdf of  $\bU$. Here we assume that $0 < \Exp(e^{U_j}) < \infty$ for $j = 1, \ldots, d$, which is equivalent to assuming that $\bR$ satisfies Condition~\ref{cond:R}. We write \GPs($\bsigma, \bgamma, F_{\bU}$) for the cdf defined by \eqref{eq:MGP:RS} and call it the $(U)$ representation. The intuition parallel to \eqref{eq:point} is that $H_U$ is the formal conditional distribution of
$$
 \bsigma \frac{e^{\bgamma(\bU - \log T)}-1}{\bgamma}
 $$
 given that $ \tfrac{\bsigma}{\bgamma}(e^{\bgamma(\bU - \log T)}-1) \nleq \bzero$ or, equivalently, given that $\bU - \log T \nleq \bzero$.

For later use we note that a  \GPs($\bone, \bzero, F_{\bU}$) vector $\bX_0$ has the cdf
\begin{equation}
\label{eq:MGP:RS0}
  H_U( \bx ) =
  \frac%
  {\int_0^\infty \, \{ F_{\bU} ( \bx +  \log t ) - F_{\bU} ( \bx \wedge \bzero +  \log t ) \} \, \diff t}%
  {\int_0^\infty \,  \bar{F}_{\bU}( \log t)  \, \diff t},
\end{equation}
and that a general \GPs\ vector $\bX$ is obtained from  $\bX_0$ through the transformation
\begin{equation}\label{eq:zerotogeneral}
  \bX = \bsigma\frac{e^{\bgamma \bX_0} - 1}{\bgamma}.
\end{equation}

Suppose now that $\bU=\bS $ where $\bS $ satisfies $\bigvee_{j=1}^d S_j=0$ and that $\bsigma =\bone$. It is straightforward to see that if $t>0$ then
$F_{\bS}(\bx + \log t) - F_{\bS}  (\bx \wedge \bzero + \log t) = \1_{\{0 < t < 1\}} \, F_{\bS}  (\bx + \log t)$, and, in particular, that $\bar{F}_{\bS}(\log t) = \1_{\{0 < t < 1\}}$. Inserting this into \eqref{eq:MGP:RS0} then gives the cdf
\begin{equation} \label{eq:FdHcdf0}
  H_S( \bx )
  =
  \int_{0}^{1} F_{\bS}(\bx + \log t) \, \diff t
  =
  \int_{0}^{\infty} F_{\bS}(\bx -v) \, e^{-v} \, \diff v,
\end{equation}
where the second equality follows from the change of variable $\log t = -v$. Further, using the transformation \eqref{eq:zerotogeneral} which connects \eqref{eq:MGP:RS0} with \eqref{eq:MGP:RS}, it follows more generally that if $\bU = \bS $, where $\bigvee_{j=1}^d S_j=0$, then for general $\bsigma, \bgamma$ one obtains the cdf
\begin{equation}\label{eq:FdH}
  H_{S}(\bx)
  =
  \int_0^{1}
    F_{\bS} \left( \tfrac{1}{\bgamma}\log \left( \tfrac{\bgamma}{\bsigma}\bx + 1 \right) + \log t \right)  \,
  \diff t = \int_0^{\infty}
    F_{\bS} \left( \tfrac{1}{\bgamma}\log \left( \tfrac{\bgamma}{\bsigma}\bx + 1 \right) - v \right) \,
    e^{-v} \,
  \diff s.
\end{equation}
We write GP$_S$($\bsigma, \bgamma, F_{\bS}$) for the cdf \eqref{eq:FdH}, and call it the $(S)$ representation.

The last expression in \eqref{eq:FdHcdf0} can be given an interpretation in terms of random variables: it is the cdf of $\bS   + E$, where $E$ is a standard exponential variable which is independent of $\bS$, and then  \eqref{eq:FdH} is the cdf of $\tfrac{\bsigma}{\bgamma}\bigl(e^{\gamma (\bS   + E)} - \bone\bigr)$. This is the \citet{ferreira2014} spectral representation transformed to the exponential scale.

\begin{theorem}
\label{thm:representation}
Suppose $\bsigma > \bzero$. The   \GPr($\bsigma, \bgamma, F_{\bR}$),  \GPs($\bsigma, \bgamma, F_{\bU}$), and \GPS($\bsigma, \bgamma, F_{\bS}$) classes defined by  \eqref{eq:MGP:R}, \eqref{eq:MGP:RS}, and \eqref{eq:FdH}, are all equal to the class of all GP distributions with $\bsigma > \bzero$.  For each class the conditional marginal distributions are given by \eqref{eq:gpmarginspos}.
\end{theorem}

\begin{proof}
The assertion for the  \GPr($\bsigma, \bgamma, F_{\bR}$) distributions follows from combining \eqref{eq:GEV-int} with \eqref{eq:MGPD}.

  By definition, the class of \GPs($\bsigma, \bgamma, F_{\bU}$) cdf-s is the same as the class of \GPr($\bsigma, \bgamma, F_{\bR}$) cdf-s, and thus the same conclusion holds for the \GPs($\bsigma, \bgamma, F_{\bU}$) cdf-s. Since \GPS($\bsigma, \bgamma, F_{\bS}$) cdf-s are \GPs($\bsigma, \bgamma, F_{\bU}$) cdf-s, it follows that they are GP distributions.

  To prove the full statement about the class \GPS($\bsigma, \bgamma, F_{\bS}$) we first note that by the construction of the \GPS($\bsigma, \bgamma, F_{\bS}$) cdf-s, it is enough to prove that the statement holds for GP distributions with $\bgamma=\bzero$ and $\bsigma = \bone$. However, then the result follows by combining (T2) with the discrete version of \cite{dehaanferreira2006}, Theorem 2.1, translated to the exponential scale (i.e., with their $\bW$ replaced by $e^{\bS +E}$), and with $\omega_0 =1$.

 The last assertion follows by straightforward calculation. As an example we prove it for the \GPr($\bsigma, \bgamma, F_{\bR}$) class for the case $\gamma_j \neq 0; \, j= 1, \ldots, d$. Let $F_j$ be the marginal distribution of the $j$-th component of $\bR$. It follows from \eqref{eq:condposmargins} and \eqref{eq:MGP:R} that $H^+_j$, the distribution of the $j$-th component of $H_R$ conditioned to be positive, is given by
$$
H_j^+(x) = 1- \frac{\int_0^{\infty} \bar{F}_j(t^{\gamma_j}(x + \frac{\sigma_j}{\gamma_j})) \, \diff t}{\int_0^{\infty} \bar{F}_j(t^{\gamma_j} \frac{\sigma_j}{\gamma_j})\, \diff t}  = 1 - (1+\tfrac{\gamma_j}{\sigma_j}x)^{-1/\gamma_j},
$$
where the second equality follows from making a change of variables from $t(\frac{\gamma_j}{\sigma_j}x + 1)^{1/\gamma_j}$ to $t$ in the numerator.
\end{proof}

It may be noted that since $F_{\bR}$, $F_{\bU}$, and  $F_{\bS}$ are cdf-s then also $H_R$, $H_U$, and  $H_S$ are cdf-s, so that in contrast to \eqref{eq:MGPD} the equations  \eqref{eq:MGP:R}, \eqref{eq:MGP:RS}, and \eqref{eq:FdH}  hold for all $\bx \in [-\infty, \infty)^d$, subject to the provision that  \eqref{eq:MGP:RS} and \eqref{eq:MGP:RS0} only apply for    $\bgamma \bx   + \bsigma> \bzero$.

The distributions of the random vectors $\bR$ and $\bU$ are not uniquely determined by the corresponding GP distributions $H_R$ and $H_U$ in \eqref{eq:MGP:R} and \eqref{eq:MGP:RS}, respectively. The next proposition is a generalization of Theorem~\ref{prop:GPproperties} (vi).

\begin{proposition}  \label{pr:undetermined}
Suppose that the random variable $Z$ is strictly positive, has finite mean and is independent of $\bR$ or $\bU$. Then \GPr$(\bsigma, \bgamma, F_{Z^{\bgamma}\bR})$ = \GPr$(\bsigma, \bgamma, F_{\bR})$ and \GPs$(\bsigma, \bgamma, F_{\bU + \log Z})$ = \GPs$(\bsigma, \bgamma, F_{\bU})$, where $Z^{\gamma_j} R_j$ should be interpreted to mean $R_j + \sigma_j\log Z$ if $\gamma_j = 0$.
\end{proposition}

\begin{proof}
We only prove the assertion for $\bR$, since the one for $\bU$ follows from it. Replacing $F_{\bR}$ by $F_{Z^{\bgamma}\bR}$ in the numerator and denominator of \eqref{eq:MGP:R} yields, after an application of Fubini's theorem and a change of variables,  a factor $\expec(Z)$ coming out in front the integrals  both in the numerator and denominator. Upon simplification, the random variable $Z$ is seen to have had no effect on $H_R$.
\end{proof}

Usually one would let the model for $\bU$ include free location parameters for each component, and the model for $\bR$ a free scale parameter for each component, in order to let data determine the relative sizes of the components.  However, as one consequence of the proposition, one should then, e.g., fix the location parameter for one of the components of $\bU$, or fix the sum of the components, to ensure parameter identifiability. Similarly, if $\gamma_j \neq 0$ for $j=1, \ldots, d$ and if the model for $\bR$ includes a free scale parameter for each component, then one should, e.g., fix one of these scale parameters.

%
%

\section{Densities, likelihoods, and censored likelihoods}
\label{sec:densities}

\subsection{Densities}
\label{sec:ssdensities}

To find the densities for the $(R)$ and $(U)$ representations, we assume that $\bR$  and $\bU$ have  densities with respect to Lebesgue measure on $\reals^d$. For the $(S)$ representation, we make the assumption that $\bS$ is obtained from a vector $\bT$  by setting $\bS = \bT - \bigvee_{j=1}^d T_{ j}$. We write \GPT($\bsigma, \bgamma, F_{\bT}$) for these distributions, write $H_T$ for the cdf-s, and call it the $(T)$ representation. Clearly, the class of \GPT\ distributions is the same as the class of \GPS\ distributions, and hence is equal to the class of all GP distributions with $\bsigma>\bzero$. The densities for the $(R)$ and $(U)$ representations are just as would be obtained if the $\bR$ and $\bU$ cdf-s were continuously differentiable and  interchange of differentiation and integration was allowed. However, they, in fact, do not require any assumptions beyond absolute continuity with respect to $d$-dimensional Lebesgue measure.

The support of the vector $\bT - \bigvee_{j=1}^d T_{ j}$ is contained in the $(d-1)$-dimensional set $\{ \bx ; \bigvee_{j=1}^d x_j =0\}$ and hence $\bT - \bigvee_{j=1}^d T_{ j}$ does not have a density with respect to Lebesgue measure on $\reals^d$. Nevertheless, the density of $H_T$ exists and can be computed if $\bT$ has a  density with respect to Lebesgue measure on $\reals^d$.


\begin{theorem} \label{th:densities}
Suppose $\bsigma > \bzero$. If $F_{\bR}$ has a density $f_{\bR}$  on $\reals^d$, then  $H_R$ has the density $h_R$ given below, if $F_{\bU}$ has density $f_{\bU}$  on $\reals^d$, then $H_U$ has the density $h_U$ below,  and if $F_{\bT}$ has density $f_{\bT}$ on $\reals^d$, then $H_T$ has density  $h_T$  below:
\begin{eqnarray} \label{eq:densityR}
  h_R(\bx) &=& \1_{\{\bx \nleq \bzero\}} \frac{1}{\int_0^\infty \bar F_{\bR}(t^{\bgamma}\tfrac{\bsigma}{\bgamma})\, \diff t}  \int_0^\infty  t^{\sum_{j=1}^{d}\gamma_j}f_{\bR}(t^{\bgamma} (\bx  + \tfrac{\bsigma}{\bgamma})) \, \diff t ,  \\
  h_U(\bx) &=& \1_{\{\bx \nleq \bzero\}}\frac{\prod_{j=1}^{d}(\gamma_j x_j +\sigma_j)^{-1}}{\int_0^\infty \bar F_{\bU}(\log t)\, \diff t}\int_0^\infty  f_{\bU}(\tfrac{1}{\bgamma}  \log(\tfrac{\bgamma }{\bsigma}\bx   + \bone) + \log t)\, \diff t, \label{eq:densityS}\\
  h_T(\bx) &=& \1_{\{\bx \nleq \bzero\}}\frac{\prod_{j=1}^{d}(\gamma_j x_j +\sigma_j)^{-1}}{  \bigvee_{j=1}^d (\frac{\gamma_j }{\sigma_j}x_j   + 1)^{1/\gamma_j}}  \int_0^\infty t^{-1} f_{\bT}(\tfrac{1}{\bgamma}  \log(\tfrac{\bgamma }{\bsigma}\bx   + \bone) +  \log t) \, \diff t, \label{eq:densityS0}
\end{eqnarray}
 for  $\bgamma \bx   + \bsigma > \bzero$, and where the densities are  $0$ otherwise. If $\gamma_j =0$ then for $h_R$  the  expressions  $t^{\gamma_j} (x_j  + \frac{\sigma_j}{\gamma_j})$  should be replaced by $x_j + \sigma_j \log t$. For $h_U$ and $h_T$, if $\gamma_j = 0$, the expressions  $\frac{1}{\gamma_j}\log(\frac{\gamma_j }{\sigma_j} x_j   + 1)$ should be replaced by their limits  $x_j/\sigma_j$.
\end{theorem}

\begin{proof}
We first prove \eqref{eq:densityR} for the special case when $\bsigma=\bgamma=\bone$, and for $\bx \nleq \bzero, \, \bx + \bone > \bzero$.  The change of variables $\by = t(\bz + \bone)$ shows that for this case
\begin{eqnarray*}
 F_{\bR}(t(\bx + \bone)) - F_{\bR}(t(\bx\wedge\bzero+ \bone)) &=&    \int \1_{\{\bm{y} \le (t(\bx + \bone),\, y \nleq t(\bx\wedge\bzero+ \bone)  \}} f_{\bR}(\by) \, \diff \by \\
   &=& \int \1_{\{\bm{z} \le \bx ,\, \bm{z} \nleq \bx\wedge\bzero  \}} t^df_{\bR}(t(\bz+\bone)) \, \diff \bz.
\end{eqnarray*}
Hence, by \eqref{eq:MGP:R}, and using Fubini's theorem,
\begin{eqnarray*}
  H_R(\bx) &=& \frac{\int_{t=0}^{\infty}\int \1_{\{\bm{z} \le \bx ,\, \bm{z} \nleq \bx\wedge\bzero  \}} t^df_{\bR}(t(\bz+\bone)) \, \diff \bz \diff t}{\int_0^\infty F_{\bR}(t\bone) \, \diff t} \\
   &=& \int  \1_{\{\bm{z} \le \bx ,\, \bm{z} \nleq \bx\wedge\bzero  \}}\frac{\int_{t=0}^{\infty} t^d f_{\bR}(t(\bz+\bone)) \, \diff t }{\int_0^\infty F_{\bR}(t\bone) \, \diff t} \diff \bz \\
   &=&\int_{(-\binfty, \bx]}  \1_{\{\bm{z} \nleq \bzero  \}}\frac{\int_{t=0}^{\infty} t^d f_{\bR}(t(\bz+\bone)) \, \diff t }{\int_0^\infty F_{\bR}(t\bone) \, \diff t} \diff \bz  .
\end{eqnarray*}
We conclude that \eqref{eq:densityR} holds for $\bsigma=\bgamma=\bone$. The proof of the general form of \eqref{eq:densityR} only differs from this case in bookkeeping details, and is omitted.

To prove \eqref{eq:densityS}, recall that $H_U = H_R$ if $\bgamma = \bzero, \,\bsigma=\bone$, so that, by \eqref{eq:densityR},
\begin{equation*}
  h_U(\bx) = \1_{\{\bx \nleq \bzero\}} \frac{\int_{t=0}^{\infty} f_{\bU}(\bx+\log t) \, \diff t}{\int_0^\infty F_{\bU}(\log t) \, \diff t}.
\end{equation*}
Writing $\tilde{H}$ for the corresponding cdf, the general cdf $H_U$ is obtained as $H_U(\bx) = \tilde{H}(\frac{1}{\bgamma} \log(\frac{\bgamma}{\bsigma}\bx + \bone)) $ and \eqref{eq:densityS} then follows by a chain rule type argument.

We again only prove \eqref{eq:densityS0} for the case $\bsigma= \bone, \bgamma = \bzero$, and with $\bx \nleq \bzero$. Also here extension to the general case is a  chain rule argument.  It follows from \eqref{eq:FdHcdf0} that
\begin{eqnarray*}
H_T(\bx) &=& \int_{t=0}^{1}  F_{\bT - \bigvee_{j=1}^d T_{ j}}(\bx + \log t)
= \int_{t=0}^{\infty}\int_{\bs} \1_{\{0\leq t \leq 1, \;  \bs - \vee_{j=0}^d s_j \leq \bx + \log t\}} f_{\bT}(\bs) \diff \bs \, \diff t.
\end{eqnarray*}
Hence, using first Fubini's theorem, then a change of variables from  $te^{\bigvee_{j=1}^d s_{j}}$ to $t$ and Fubini's theorem, and finally a change of variables from  $\bs$ to  $\bs + \log t$ and Fubini's theorem,
\begin{eqnarray*}
 H_T(\bx)  &=& \int_{\bs} \int_{t=0}^{\infty}e^{-\bigvee_{j=1}^d s_{j}} \1_{\{0 \leq t e^{-\bigvee_{j=1}^d s_j} \leq 1, \; \bs\leq \bx + \log (t)\}}f_{\bT}(\bs)\, \diff t \, \diff \bs \\
   &=& \int_{t=0}^{\infty} \int_{\bs} e^{-\bigvee_{j=1}^d s_{j}} t^{-1} \1_{\{ e^{-\bigvee_{j=1}^d s_j} \leq 1, \; \bs \leq \bx  \}}f_{\bT}(\bs + \log t) \, \diff \bs \, \diff t \\
   &=& \int_{- \binfty}^{\bx} \1_{\{\bs \nleq \bzero\}} \, e^{-\bigvee_{j=1}^d s_{j}}  \int_{t=0}^{\infty}t^{-1} f_{\bT}(\bs + \log t)  \, \diff t \, \diff \bs.
\end{eqnarray*}
This proves that \eqref{eq:densityS0} holds for $\bgamma=\bzero$ and $\bsigma=\bone$.
\end{proof}


In some cases, the integrals in \eqref{eq:densityR}, \eqref{eq:densityS} and \eqref{eq:densityS0} can be computed explicitly; see the examples below.  Otherwise the one-dimensional integrals allow for fast numerical computation as soon as one can compute densities and distribution functions of $\bR$ or $\bU$ efficiently. Either way, this can make full likelihood inference possible, also in high dimensions.

\subsection{Censored likelihood}
\label{sec:censored}

Sometimes one does not trust the GP distribution to fit the excesses well on the entire set $\bx \nleq \bzero$. Then, instead of using a full likelihood obtained as a product of the densities in Theorem~\ref{th:densities}, one can use a censored likelihood which is based on the values of the excesses which are larger than some censoring threshold $\bv = (v_1, \ldots, v_d)$. This idea was introduced for multivariate extremes in \citet{smith+tawn+coles:1997}, and has since become a standard approach to inference. \citet{huser+d+g:2015} explore the merits of this and other approaches via simulation.

Write $D=\{1, \ldots, d\}$, and let $C \subset D$ be the set of indices which correspond to the components which are censored, i.e., which do not exceed their censoring threshold $v_i$. Then, using the notation $\bx_A = \{x_j; j \in A \}$ and writing $h$ for $h_R, h_U$ or $h_T$, the likelihood contribution of a censored observation is
\begin{equation}\label{eq:censored}
 h_C(\bx_{D \setminus C}) = \int_{\{x_j \in (-\infty, v_j]; \, j \in C \}} h(\bx) \, \diff \bx_C.
\end{equation}

For certain models, the $|C|-$dimensional integral in equation~\eqref{eq:censored} can be avoided, which is advantageous from a practical perspective.

\begin{example}
\label{ex:nuggets}
 The simplest situation is when the components of the shape vector $\bR$  are mutually independent. This could e.g.\ be a model for windspeeds over a small area, perhaps a wind farm, with $T^{-\bgamma}$ representing the intensity of the average geostrophic wind and with the components of $\bR$ representing random wind variations caused by local turbulence.

 Let $f_j$ be the density function of $R_j$, the $j$th component of $\bR$, let $F_j$ be the corresponding cdf, write $y_j = x_j+\sigma_j/\gamma_j$, and assume that $v_j \leq 0, j \in C$. The integral which appears in the numerator in \eqref{eq:censored} for $h=h_R$ in \eqref{eq:densityR} can then be written as
$$
\1_{\{\bx_{D \setminus C} \nleq \bzero \}  }  \int_{0}^{\infty} t^{\sum_{j \in D\setminus C} \gamma_j} \prod_{j \in C}F_j(t^{\gamma_j} v_j)\prod_{j \in D \setminus C}f_j(t^{\gamma_j} y_j) \, \diff t
$$
  and the integral in the denominator equals $\int_{0}^{\infty}
\{1- \prod_{j=1}^d F_j (t^{\gamma_j} \sigma_j/\gamma_j)\}\, \diff t$.
Here quick numerical computation of both integrands is typically possible.

Sometimes these integrals can also be computed analytically, and similarly for the corresponding integrals for $h_U$ and $h_T$. As a simple example, consider \eqref{eq:densityS0} with $\bgamma = \bzero$ and $\bsigma=\bone$ and with the components of $\bT$  having independent standard  Gumbel distributions with cdf $\exp\{-e^{-x}\}$. Then, with $c$  the number of elements in $C$, i.e., the number of censored components, and abbreviating $\1_{\{\bx_{D \setminus C} \nleq \bzero \}  }$ to $\1_{D\setminus C}$, we obtain that
\begin{eqnarray*}
 h_C(\bx_{D\setminus C}) &=&    \1_{D\setminus C} \,   e^{-\bigvee_{j=1}^d x_j} \int_{0}^{\infty} \prod_{j \in D\setminus C}e^{-x_j - \log t} \exp\{-e^{-x_j - \log t}\} \prod_{j \in C}\exp\{-e^{-v_j - \log t}\}\, \diff t\\
   &=& \1_{D\setminus C} \,  e^{-\bigvee_{j=1}^d x_j - \sum_{j \in D\setminus C} x_j}\int_{0}^{\infty} t^{-(d-c)} \exp\{-t^{-1} \,\bigl(\sum_{j \in D\setminus C} e^{- x_j} + \sum_{j \in C} e^{- v_j}\bigr) \} \, \diff t\\
   &=& \1_{D\setminus C} \, (d-c-2)! \, e^{-\bigvee_{j=1}^d x_j - \sum_{j \in D\setminus C} x_j} \bigl(\sum_{j \in D\setminus C} e^{- x_j} + \sum_{j \in C} e^{- v_j}\bigr)^{-(d-c)+1}.
\end{eqnarray*}
\end{example}

Whilst the previous example is a theoretical illustration, the class of GP distributions obtained by letting $\bR$ (or $\bU$) have independent components with parametrized marginal distributions does make for a large and flexible class of models. It includes, for example, the GP distributions associated to the commonly used logistic and negative logistic max-stable distributions. For this and further examples, see \citet{kiriliouk+rootzen+segers+wadsworth:2016}.

\subsection{Further examples}
\label{sec:examples}

We illustrate two further constructions with tractable densities. The first is a toy example to exhibit the idea of building process knowledge into a model. The second is a variation on existing extreme value models based on lognormal distributions.

\begin{example}
\label{ex:rivers}
An extreme flow episode in a river network consisting of two tributaries which join to form the main river could be modeled as $ \bR/T^{\bgamma} = (R_1/T^\gamma, R_2/T^\gamma, (R_1+R_2+E)/T^\gamma)$, with $\gamma > 0$, so that $R_3 = R_1 + R_2 + E$. Here the first component corresponds to flow in tributary number one, the second component to flow in tributary number two, and the third component to flow in the main river. The simplest model is that $R_1, R_2, E$ are independent and have a standard exponential distribution. Then,
\begin{eqnarray}
\label{eq:river1}
  \int_{0}^{\infty} t^{3\gamma}f_{\bR}(t^{\bgamma}\by) \, \diff t &=&  \1_{\{\bzero \leq \by, \, y_1 + y_2 \leq y_3\}} \int_{0}^{\infty}t^{3\gamma} e^{-t^\gamma y_3}\, \diff t \,  \nonumber \\
   &=& \1_{\{\bzero \leq \by, \,y_1 + y_2 \leq y_3\}}\, \gamma^{-1}\Gamma(3+1/\gamma) \, y_3^{-3-1/\gamma}.
\end{eqnarray}
Assuming in addition that $\bsigma =(\sigma, \sigma, \sigma)$, we have
\begin{eqnarray*}
   \int_{0}^{\infty} \bar{F}_{\bR}(t^{\bgamma} \bsigma/\bgamma) \, \diff t
   &=& \expec[\textstyle{\bigvee_{j=1}^3} (R_j \, \gamma/\sigma)^{1/\gamma}] \\
   &=& (\gamma/\sigma)^{1/\gamma} \expec[R_3^{1/\gamma}] \\
   &=& (\gamma/\sigma)^{1/\gamma}\Gamma(3+1/\gamma)/2,
\end{eqnarray*}
since $R_3$ is a sum of three exponential variables, and thus has a gamma distribution. It follows from \eqref{eq:densityR} and \eqref{eq:river1} with $\by = \bx + \tfrac{\bsigma}{\bgamma}$ that
\begin{eqnarray*}
 h_R(\bx) = \1_{\{\bx \nleq \bzero, \, -\bsigma/\bgamma \leq \bx, \, x_1 + x_2 \leq x_3\}}\, \gamma^{-1} 2 (\sigma/\gamma)^{1/\gamma}   \,  (x_3 + \sigma/\gamma)^{-3-1/\gamma}.
\end{eqnarray*}
\end{example}

\smallskip

\begin{example}
\label{ex:loggaussian}
Lognormal distributions have been used in max-stable modelling, e.g., in  \citet{huser2013}, and as point process models in \citet{wadsworth2014}, and are an important class of models. As an example, in the $(R)$ representation, suppose that $\bzero \leq \bgamma$ and that  $F_{\bR}(\bx) = \Phi(\log \bx)$, where $\Phi$ is the cdf of a $d$-dimensional normal distribution with mean $\bmu$ and nonsingular covariance matrix $\Sigma$. Write $\phi$ for the corresponding density and let $A=\Sigma^{-1}$ be the precision matrix. Then, writing $\by = \log(\bx+\tfrac{\bsigma}{\bgamma}) - \bmu$, we have
\begin{multline*}
  \int_0^\infty  t^{\sum_{j=1}^{d}\gamma_j}f_{\bR} \left( t^{\bgamma} (\bx  + \tfrac{\bsigma}{\bgamma}) \right) \, \diff t
  =
  \frac{1}{\prod_{j=1}^{d}(x_j + \frac{\sigma_j}{\gamma_j})}
  \int_0^\infty  \phi \left(\bgamma \log t+ \log(\bx+\tfrac{\bsigma}{\bgamma}) \right) \, \diff t \\
  =
  \frac{1}{\prod_{j=1}^{d}(x_j + \frac{\sigma_j}{\gamma_j})}
  \frac{|A|^{1/2}}{(2\pi)^{d/2}}
  \int_0^\infty   \exp \left( -\tfrac{1}{2}(\bgamma \log t+ \by) A(\bgamma \log t+ \by)' \right) \, \diff t.
\end{multline*}
Making the change of variables from $\log t$ to $t$ and completing the square, we can evaluate the integral, finding
\begin{equation*}
  h_R(\bx)
  =
  \1_{\{\bx \nleq \bzero\}}
  \frac{|A|^{1/2}}{[(2\pi)^{(d-1)} \, \bgamma A \bgamma']^{1/2}}
  \frac{1}{\prod_{j=1}^{d}(x_j + \frac{\sigma_j}{\gamma_j})}
  \frac%
    {\exp \left[ -\frac{1}{2} \left( \by A \by' -  \frac{(\bgamma A \by'-1)^2}{\bgamma A \bgamma'} \right) \right]}%
    {\int_0^\infty \bar \Phi \left( \bgamma\log(t) +\log(\bsigma/\bgamma) \right) \, \diff t}.
\end{equation*}
The integral in the denominator can be expressed as a sum of $d$ components, each of which involves a $(d-1)$-dimensional normal cdf, see \cite{huser2013}. However, if $d$ is large then this expression is cumbersome. Inference methods for similar high-dimensional models are explored in \citet{deFondevilleDavison2016}.

\end{example}

\section{Probabilities and conditional probabilities}
\label{sec:condprob}
Equations \eqref{eq:MGP:R}, \eqref{eq:MGP:RS}, and \eqref{eq:FdH} give probabilities of rectangles for GP distributions, on the real scale. In this section they are  generalized to expressions for probabilities of general sets and for conditional probabilities.
Below, we only consider \GPr\ models. It is straightforward to derive the corresponding formulas for the other representations.

Let  $F = \{\by; \by \nleq \bzero \}$, set $A = \{\by; \, \by \leq \bx\}$, and for $\ba, \bb \in \reals^d$ and a set $B \subset \reals^d$ write $\ba(B+\bb)$ for the set $\{\ba(\by +\bb); \, \by \in B\}$. As is easily checked,
\begin{equation}
\label{eq:formalcond}
  H_R(\bx)
  =
  H_R(A)
  =
  \frac%
    {\int_{0}^{\infty} \Pr[\bR \in t^{\bgamma} (A \cap F + \bsigma/\bgamma)] \, \diff t  }%
    {\int_{0}^{\infty} \Pr[\bR \in t^{\bgamma} ( F + \bsigma/\bgamma)]\, \diff t}.
\end{equation}
Now, if in the derivation of \eqref{eq:MGP:R} the special set $A$  defined above is replaced by a general set $A \subset \reals^d$, the result still is the same,
\begin{equation}
\label{eq:formalcondA}
  H_R(A)
  =
  \frac%
    {\int_{0}^{\infty} \Pr[\bR \in t^{\bgamma}( A \cap F + \bsigma/\bgamma)] \, \diff t }%
    {\int_{0}^{\infty} \Pr[\bR \in t^{\bgamma}( F + \bsigma/\bgamma)] \, \diff t}
  =
  \frac%
    {\int_{0}^{\infty} \Pr[\bR/t^{\bgamma} - \bsigma/\bgamma  \in A \cap F ] \, \diff t }%
    {\int_{0}^{\infty} \Pr[\bR/t^{\bgamma} - \bsigma/\bgamma \in F ] \, \diff t}.
\end{equation}
A proof that \eqref{eq:formalcondA} holds for any set $A$ is immediate: using Fubini's theorem it is seen that the right-hand side of the equation is a probability distribution as function of $A$, and since it agrees with the distribution $H_R$ on sets of the form $\{\by; \, \by \leq \bx\}$, the two distributions are equal. The intuition is that $H_R(A)$ is the (formal) conditional probability of the event $\{\bR/T^{\bgamma} - \bsigma/\bgamma \in A\}$  given the event $\{\bR/T^{\bgamma} - \bsigma/\bgamma \nleq \bzero\}$.

Let the random vector $\bX$ have the distribution $H_R$ in \eqref{eq:MGP:R}. Then $\Pr[\bX \in A \mid \bX \in B] = H_R(A\cap B)/H_R(B)$, and hence \eqref{eq:formalcondA}  can also be used to find conditional probabilities. Further, assuming continuity, \eqref{eq:densityR} determines the conditional densities. For instance, writing $f_{|\, X_1 =x}$ for the conditional density of $(X_2, \ldots X_d)$ given that  $X_1 =x$, we find, for $x>0$,
\begin{equation} \label{eq:conddensity}
f_{|\, X_1 =x}(x_2, \ldots x_d)
  =
  \frac%
    {\int_0^\infty  t^{\sum_{j=1}^{d}\gamma_j}f_{\bR}\left(t^{\bgamma} ((x, x_2, \ldots x_d)  + \tfrac{\bsigma}{\bgamma}) \right) \, \diff t}%
    {\int_0^\infty  t^{\gamma_1}f_{R_1} \left( t^{\gamma_1} (x  + \frac{\sigma_1}{\gamma_1}) \right) \, \diff t}.
\end{equation}
By further integration, it follows that
\begin{multline}\label{eq:condprobA}
  \Pr[\bX \in A \mid X_1 = x] \\
  =
  \frac%
    {\int_0^\infty
      t^{\gamma_1}
      f_{R_1} \left(t^{\gamma_1} (x  + \frac{\sigma_1}{\gamma_1})\right)
      \Pr[ (x, R_2, \ldots R_d)/t^{\bgamma}- \tfrac{\bsigma}{\bgamma} \in   A  \mid R_1 =t^{\gamma_1} (x +\frac{\sigma_1}{\gamma_1})] \,
      \diff t}%
    {\int_0^\infty  t^{\gamma_1}f_{R_1} \left( t^{\gamma_1} (x  + \frac{\sigma_1}{\gamma_1})\right) \, \diff t}.
\end{multline}

\begin{example}
\label{ex:rivers2}
In Example \ref{ex:rivers}, extreme flow episodes in the two river tributaries are modelled using $(R_1/T^\gamma, R_2/T^\gamma)$ with $R_1$ and $R_2$ independent  standard  exponential variables and with $\gamma >0$. Suppose $\bX \sim H$ where $H$ is the GP distribution obtained from $(R_1/T^\gamma, R_2/T^\gamma)$ and let $s>0$. Since $R_1 + R_2$ has a gamma distribution, it is straightforward to evaluate  \eqref{eq:formalcondA}  to find the distribution of the sum of the flows in the two tributaries:
\begin{equation}
\label{eq:riversum}
  \Pr[ X_1+X_2 > s] = c_1 \left(1+ \tfrac{\gamma}{\sigma_1 + \sigma_2} s\right)^{-1/\gamma},
\end{equation}
with $c_1=\gamma^{-1}(1+\gamma)(\sigma_1 + \sigma_2)^{-1/\gamma} / [\sigma_1^{-1/\gamma} + \sigma_2^{-1/\gamma} - (\sigma_1 + \sigma_2)^{-1/\gamma}] $.

Similar computations using \eqref{eq:conddensity} show that for $x_1, x_2>0$
\begin{equation*}
  f_{|\, X_1 =x_1}(x_2)=c_2 \left( 1 + \frac{\gamma/(1+\gamma )}{(\gamma x_1 +\sigma_1 +\sigma_2)/(1+\gamma )}\, \, x_2 \right)^{-1-1/[\gamma/(1+\gamma)]}
\end{equation*}
and
\begin{equation*}
  \Pr[X_2>x_2 \mid X_1=x_1] = c_3 \left( 1 + \frac{\gamma/(1+\gamma )}{(\gamma x_1 +\sigma_1 +\sigma_2)/(1+\gamma )}\, \, x_2 \right)^{-1/[\gamma/(1+\gamma)]},
\end{equation*}
for $c_2=(1+\gamma)(x_1+\sigma_1/\gamma)^{1+1/\gamma}(\gamma x_1 +\sigma_1 + \sigma_2)^{-2-1/\gamma}$  and $c_3 = (x_1+\sigma_1/\gamma)^{1+1/\gamma}(\gamma x_1 +\sigma_1 + \sigma_2)^{-1-1/\gamma}$. Hence, dividing \eqref{eq:riversum} with the same expression with $s$ set to zero, we find that the sum conditioned to be positive has a GP distribution with the same shape parameter as the marginal distributions but with a larger scale parameter. The conditional distribution of $X_2$ given that $X_1=x>0$, conditioned to be positive, has a GP distribution with a smaller shape parameter, $\gamma/(1+\gamma)$.
\end{example}

Many of the results in Example~\ref{ex:rivers2} hold more generally. For instance, the conditional GP  distribution of sums holds as soon as the marginals have the same shape parameter. The intuition is simple: Suppose the $\GP_R$ distribution has been  obtained from the vector $(R_1/T^\gamma - \sigma_1/\gamma, \ldots, R_d/T^\gamma - \sigma_d/\gamma)$  by (formal) conditioning on at least one of the components being positive. Then a weighted sum of the components equals $R/T^\gamma - \sigma/ \gamma$, for $R=\sum_{j=1}^{d} a_j R_j$ and $\sigma=\sum_{j=1}^{d}a_j\sigma_j$, with coefficients $a_1, \ldots, a_d$.  According to the $\GP_R$ representation, provided $a_1, \ldots, a_d \geq 0$, the distribution of $R/T^\gamma - \sigma/ \gamma$ conditioned to be positive is a one-dimensional GP distribution with parameters $\gamma$ and $\sigma$. Further, that a sum is positive implies that at least one component is positive, and hence first conditioning on at least one component being positive, and then conditioning on the sum being positive gives the same result as conditioning directly on the sum being positive. Thus the one-dimensional GP distribution holds for component sums in GP distributions. Similar reasoning can be applied to, e.g., joint distributions of several weighted sums and several components. Here, we only prove the one-dimensional result.

\begin{proposition}
\label{prop:sum}
Let $\bX$ be a GP random vector with common shape parameter $\gamma$ for all $d$ margins and with scale parameter $\bsigma > \bzero$, and if $\gamma \leq 0$ additionally assume that $\Pr[\sum_{j=1}^d a_j X_j>0] >0$.  Then, for $\ba \in [ \bzero, \binfty) \setminus \{ \bzero \}$, the conditional distribution of the weighted sum $\sum_{j=1}^d a_j X_j$ given that it is positive is generalized Pareto with shape parameter $\gamma$ and scale parameter $\sigma = \sum_{j=1}^d a_j \sigma_j$.
\end{proposition}

\begin{proof}
Since $\Pr[\sum_{j=1}^d a_j X_j>0 ] >0$ holds automatically if $\gamma> 0$ and $\bsigma> \bzero$, this condition is satisfied for all values of $\gamma$. Let $A_x=\{\by \in \reals^d \, | \,  \sum_{j=1}^d a_j y_j>x \}$ and as above define $R=\sum_{j=1}^{d} a_j R_j$. Then, for $x>0$ and with $F = \{\by; \by \nleq \bzero \}$, as above, $A_x \cap F = A_x$, and [for $\gamma=0$ using the convention that $t^0 (x+ \sigma/0)$ means $x+\sigma  \log t$] the numerator in \eqref{eq:formalcond}  for $A=A_x$ is
\begin{equation*}
  \int_{0}^{\infty} \Pr[\bR/t^{\bgamma} - \bsigma/\bgamma \in A_x] \, \diff t
  =
  \int_{0}^{\infty} \Pr[R/t^\gamma - \sigma/\gamma > x] \, \diff t,
\end{equation*}
and hence  by \eqref{eq:formalcond}
\begin{eqnarray*}
  \textstyle \Pr \left[  \sum_{j=1}^d a_j X_j>x \, \Big| \, \sum_{j=1}^d a_j X_j>0 \right]
  &=&
  \frac{H_R(A_x)}{H_R(A_0)}
  =
  \frac%
    {\int_{0}^{\infty} \Pr[ R/t^\gamma - \sigma/\gamma > x] \, \diff t}%
    {\int_{0}^{\infty} \Pr[ R/t^\gamma - \sigma/\gamma > 0] \, \diff t}\\
  &=&
  (1+\tfrac{\gamma}{\sigma}x)^{-1/\gamma},
\end{eqnarray*}
where the last equality follows from making the change of variables from $t(1+x/\sigma)^{1/\gamma}$ to $t$ in the numerator.
\end{proof}

Example~\ref{ex:independentFrechet} exhibits a situation where the component sum in a GP distribution is identically equal to $- \infty$  and hence the assumption $\Pr[ X_1 + X_2 > 0] > 0$ is not satisfied.
%
%
\section{Simulation}
\label{sec:simulation}

In this section we outline four methods for sampling from multivariate GP distributions. For Methods 1 to 3 we focus on simulation of a GP vector $\bX_0$ with $\bsigma=\bone$ and $\bgamma=\bzero$, since a vector $\bX$ with general $\bsigma$ and $\bgamma$ is obtained at once from the vector $\bX_0$ through \eqref{eq:zerotogeneral}. Furthermore, using the connection between \GPs\ and \GPr\ distributions, \GPr\ vectors may be obtained by simulating \GPs\ vectors, and vice versa. Throughout we assume that simulation of $\bU$ from $F_{\bU}$ and $\bT$ from $F_T$ is possible. Recall the relation $\bS = \bT - \bigvee_{j=1}^d T_j$ which was used to define the $(T)$ representation. The first method follows immediately from \eqref{eq:FdHcdf0}.

\medskip

\noindent
{\em Method 1: simulation from the $(T)$ representation.} Simulate a vector $\bT \sim F_{\bT}$ and an independent variable $E \sim $ Exp$(1)$ and set $\bX_0 = E + \bT-\max_{1\leq j \leq d}T_j$.

\medskip

Simulation from the $(R)$ and $(U)$ representations is less direct. We propose three methods: rejection sampling, MCMC sampling, and approximate simulation using \eqref{eq:smgpdappr}. The idea in Methods~2 and~3 is to use an appropriate change of measure so that Method~1 can be used to simulate from the $(T)$ representation. The \GPT($\bone, \bzero, F_T$) density is
\begin{equation*}
  h_T(\bx) = \1_{\{\bx \nleq \bzero\}} \,  e^{-\bigvee_{j=1}^d x_j}
 \int_{0}^{\infty} t^{-1} f_{\bT}(\bx + \log t) \, \diff t.
\end{equation*}
If in this equation one replaces $\bT$  by $\bT_0$ where $\bT_0$ has density
\begin{equation}
\label{eq:tilteddensity}
  f_{\bT_0}(\bx) =  \frac{e^{\bigvee_{j=1}^d x_j} f_{\bU}(\bx)}{\int_{0}^{\infty} \bar{F}_{\bU}(\log t) \, \diff t}
\end{equation}
then
\begin{eqnarray*}
  h_T(\bx) &=& \1_{\{\bx \nleq \bzero\}} \frac{e^{-\bigvee_{j=1}^d x_j}\int_0^{\infty} t^{-1} e^{\bigvee_{j=1}^d (x_j+\log t)} f_{\bU}(\bx + \log t) \, \diff t}{\int_{0}^{\infty} \bar{F}_{\bU}(\log t) \, \diff t}\\
  &=& \1_{\{\bx \nleq \bzero\}} \frac{\int_0^{\infty}
   f_{\bU}(\bx + \log t) \, \diff t}{\int_{0}^{\infty} \bar{F}_{\bU}(\log t) \, \diff t} = h_U(\bx).
\end{eqnarray*}
Thus, if one can simulate $\bT_0$ vectors, then these give \GPs($\bone, \bzero, F_U$) vectors via Method~1.

\medskip

\noindent
{\em Method 2: simulation of $\bT_0$ via rejection sampling.} Let $\varphi$ be a probability density function which satisfies $f_{\bT_0}(\bx) \leq K \varphi(\bx)$, for some constant $K > 0$. Draw a candidate vector $\bT_0^c$ from $\varphi$ and accept the candidate with probability $f_{\bT_0}(\bm{T}^c_0) / [K \varphi(\bm{t}^c_0)]$, and repeat otherwise. Use the accepted vector as input $\bT$ in Method~1.

\medskip

The acceptance probability is $1/K$, and thus it is advantageous to find a $\varphi$ such that $K$  is not too large. In high dimensions however, such a $\varphi$ might be difficult to find.

\medskip

\noindent
{\em Method 3: simulation of $\bT_0$ via MCMC.} Use a standard Metropolis--Hastings algorithm to simulate from a Markov chain with stationary distribution \eqref{eq:tilteddensity}. At iteration $i$, draw a candidate vector $\bT_0^c$ from the density $f_{\bT}$ and accept the candidate with probability $\min\{1,  \exp(\bigvee_{j=1}^d T^{c}_{0,j} -\bigvee_{j=1}^d T^{i-1}_{0,j}) \}$,  where $\bT^{i-1}_{0}$ is the current state of the chain. If the candidate is not accepted, then the previous state of the chain is repeated. After a suitable burn-in time, values of the chain should represent dependent samples from~\eqref{eq:tilteddensity}; the draws can be thinned to produce approximately independent replicates. Use the simulated values of the chain as inputs $\bT$ to Method~1.

\medskip
Alternative proposal distributions could be used with appropriate modification of the acceptance probability; for details see e.g.\ \citet{chib-greenberg1995}.

By \eqref{eq:smgpdappr}, an approximate way to simulate $\bX \sim$ \GPr($\bsigma, \bgamma, F_{\bR}$) is as follows.

\medskip

\noindent
{\em Method 4: approximate simulation from the $(R)$ representation.}  Choose a large $K > 0$. Simulate $\bar T \sim $ Unif$\,[0, K]$ and an independent $\bR \sim F_{\bR}$. If $\bR/\bar{T}^{\bgamma} \nleq \bsigma/\bgamma$ set $\bX = \bR/\bar{T}^{\bgamma} - \bsigma/\bgamma$, and repeat otherwise.

\medskip

In this algorithm, the probability to keep a simulated $\bR/\bar{T}^{\bgamma}$ value is
$$
\frac{1}{K} \int_0^K \bar{F}_{\bR}(t^{\bgamma}\tfrac{\bsigma}{\bgamma}) \, \diff t \approx \frac{1}{K}
\int_0^\infty \bar{F}_{\bR}(t^{\bgamma}\tfrac{\bsigma}{\bgamma}) \, \diff t,
$$
so  one has to simulate approximately $K /\int_0^\infty \bar{F}_{\bR}(t^{\bgamma}\tfrac{\bsigma}{\bgamma}) \, \diff t$ values of $\bR/\bar{T}^{\bgamma}$ to get one $\bX$-value. Hence a large $K,$ which ensures that the approximating distribution $H^{(K)}$ is close to $H,$ leads to longer computation times, and a compromise has to be made. As a guide to the compromise, it is often, e.g.\ for Gaussian or log-Gaussian processes, possible to compute, analytically or numerically, sharp bounds for the approximation errors.

To summarize, Method 1 is simplest, but only produces \GPT($\bsigma, \bgamma, F_T$) vectors. Method~2 and Method~3 provide ways to simulate vectors $\bT_0$ from distribution~\eqref{eq:tilteddensity}, which can then be inserted into Method~1 to simulate from the \GPs($\bsigma, \bgamma, F_{\bU}$) and \GPr($\bsigma, \bgamma, F_{\bR}$) distributions. Method~4 is as simple to program as Method~1  and produces i.i.d.\ vectors, but, similarly to Method~3,  only approximates the target distribution.

\section{Conclusion}
\label{sec:conclusion}
This paper studies the probability theory underlying peaks over thresholds modelling of multivariate  data using generalized Pareto distributions. We first derive basic properties of the multivariate GP distribution, including behaviour under conditioning; scale change;   convergence in distribution; mixing; and connections with generalized extreme value distributions.  The main results are a point process limit result which gives a general and concrete description of the behaviour of extreme episodes; new representations of the cdf-s of multivariate GP distributions, motivated by and derived from the point process result; expressions for likelihoods and censored likelihoods; formulas for probabilities and conditional probabilities of general sets; and algorithms for random sampling from multivariate GP distributions. Throughout, the results are illustrated by examples.

We provided four different representations of GP distributions, labelled $(R)$, $(S)$, $(T)$, and $(U)$. Computationally, the $(T)$ densities are simplest, and simulation from the $(T)$ representation also is simpler than simulation from the other representations. On the other hand, it seems impractical to use the $(T)$ representation for prediction or spatial modelling, since taking  lower-dimensional margins of it do not simply lead to the proper lower-dimensional $(T)$ representations, and since a $d$-dimensional $(T)$ representation does not include any prescription for how to extend it to a $(d+1)$-dimensional one. The $(S)$, $(T)$ and $(U)$ representations allow for smooth transitions from positive to negative $\gamma_j$, in contrast to the $(R)$ representation. In some situations, however, requirements of realistic physical modelling can nevertheless lead to the use of the $(R)$ representation.

Peaks over thresholds modelling  of extremes of a random vector $\bY$ first selects a suitable level $\bu$ and then models the distribution of the over- and undershoots, $\bX= \bY-\bu$,  conditional on the occurrence of at least one overshoot, by a GP distribution. Of course, this GP model also  models  the conditional distribution of the original vector $\bY$, since $\bY=\bX+\bu$. Modelling issues which are not treated include choice of the level $\bu$, perhaps as a function of covariates like time, and modelling of the Poisson process which governs the occurrence of extreme episodes.

A further practical issue, which is outside the scope of the current paper, is that of asymptotic independence of extremes. In the event that the limiting probability of joint occurrence of extremes, conditional upon at least one extreme component, is zero, multivariate GP distributions will typically not represent the best models. Asymptotic independence is usually manifested in practice by the threshold stability properties of multivariate GP distributions not holding. Diagnostics based on these stability properties are presented in \citet{kiriliouk+rootzen+segers+wadsworth:2016}.

The paper gives a basis for understanding and modelling  of extreme episodes. We believe it will contribute to the solution of many different and important risk handling problems. However, it still is an early excursion into new territory, and much research remains to be done. Important challenges include incorporating temporal dependence and developing methods for prediction  of the unfolding of extreme episodes.

\appendix

\section*{Appendix}
\label{sec:proofs}

\begin{proof}[Proof of \eqref{eq:MGPD}]
If $x_j < \eta_j$ for some $j \in \{1, \ldots, d\}$, then $\nu( \{ \by : \by \le \bx \} ) = 0$, so that $H( \bx ) = 0$ too. Let $\bx > \bEta$. We have
\begin{align*}
  \{ \by \;:\; \by \not\leq \bzero, \; \by \leq \bx \}
  &= \{ \by \;:\; \exists j, y_j > 0; \forall j, y_j \le x_j \} \\
  &= \{ \by \;:\; \exists j, y_j > x_j \wedge 0; \forall j, y_j \le x_j \} \\
  &= \{ \by \;:\; \by \not\leq \bx \wedge \bzero, \; \by \leq \bx \}.
\end{align*}
As a consequence,
\begin{align*}
  H(\bx)
  &=
  \frac%
    { \nu( \{ \by : \by \not\leq \bzero, \; \by \leq \bx \} ) }%
    { \nu( \{ \by : \by \not\leq \bzero \} ) } \\
  &=
  \frac%
    { \nu( \{ \by : \by \not\leq \bx \wedge \bzero, \; \by \leq \bx \} ) }%
    { \nu( \{ \by : \by \not\leq \bzero \} ) } \\
  &=
  \frac%
    {(-\log G( \bx \wedge \bzero )) - (-\log G(\bx))}%
    {-\log G(\bzero)}
  =
  \frac{1}{\log G(\bzero)}
  \log \left( \frac{G(\bx \wedge \bzero)}{G(\bx)} \right),
\end{align*}
as required.
\end{proof}

The following property was used in the course of the proof of Theorem~\ref{prop:GPproperties}(vi).

\begin{proof}[Proof: a GEV cdf $G$ with $\bsigma > \bzero$ is determined by its values for $\bx \geq \bzero$]
Since $\bsigma > \bzero$  the margins of $G$ has the form \eqref{eq:gmarginals}, and hence $\bsigma$ and $\bgamma$ are determined by the values of $G(\bx)$ for $\bx > \bzero$. Further, by max-stability we have that $G(\ba_t \bx + \bb_t)^t = G(\bx)$ and hence $G(\bx)$ is determined for all values of $\bx$ such that $\ba_t \bx + \bb_t \geq \bzero$ i.e. for $\bx \geq - \ba_t^{-1} \bb_t$. Using \eqref{eq:tparameters}, it is seen that if $\gamma_i > 0$ then $-a_{t, i}^{-1} b_{t, i}\to - \sigma_i/\gamma_i =\eta_i $ and if $\gamma_i= 0$ then $- a_{t, i}^{-1} b_{t, i}\to  -\infty = \eta_i$ as $t \to \infty$. Further, if $\gamma_i < 0$ then $-a_{t, i}^{-1} b_{t, i}\to - \infty =\eta_i $ as $t \to 0$. Thus $G(\bx)$ is determined for all values in the support of $G$, and this in turn determines $G(\bx)$ for all values of $\bx$.
\end{proof}


\section*{Acknowledgement}

\noindent
Anna Kiriliouk has throughout participated in the discussions leading to this paper. We thank her and two referees for many helpful comments. Research supported by the Knut and Alice Wallenberg foundation. Johan Segers was funded by contract ``Projet d'Act\-ions de Re\-cher\-che Concert\'ees'' No.\ 12/17-045 of the ``Communaut\'e fran\c{c}aise de Belgique'' and by IAP research network Grant P7/06 of the Belgian government.

\bibliographystyle{chicago}
{\small
\bibliography{libraryMGPD}
}
\end{document}